\documentclass[11pt, oneside]{amsart}
\usepackage[text={6.5in,9in},centering,letterpaper,dvips]{geometry}
\usepackage{graphicx}
\usepackage{amsfonts}
\usepackage[dvipsnames]{xcolor}
\usepackage{subcaption}
\usepackage{epsf}
\usepackage{amssymb}
\usepackage{amsmath}
\usepackage{amscd}
\usepackage{pdfpages}
\usepackage{fancyhdr}
\usepackage{setspace}
\usepackage{soul} 
\usepackage[all]{xy}
\usepackage{verbatim}
\usepackage{enumerate}
\usepackage[colorlinks=true, urlcolor=NavyBlue, linkcolor=NavyBlue, citecolor=NavyBlue,backref=page]{hyperref}
    \usepackage{tikz}
\usetikzlibrary{  cd,  calc, positioning,  fit, arrows,  decorations.pathreplacing, decorations.markings,  shapes.geometric, backgrounds,bending}
\usepackage{tikzsymbols}

\renewcommand*{\backref}[1]{}
\renewcommand*{\backrefalt}[4]
{%
    \ifcase #1 (Not cited.)%
        \or        (Cited on page~#2.)
        \else      (Cited on pages~#2.)
    \fi
}

\makeatletter

\renewcommand{\p@subfigure}{\thefigure.} 
\makeatother

\theoremstyle{theorem}
\newtheorem{theorem}{Theorem}[section]
\newtheorem{proposition}[theorem]{Proposition}

\newtheorem{question}[theorem]{Question}
\newtheorem{corollary}[theorem]{Corollary}

\makeatletter
\newtheorem*{rep@theorem}{\rep@title}
\newcommand{\newreptheorem}[2]{%
\newenvironment{rep#1}[1]{%
 \def\rep@title{#2 \ref{##1}}%
 \begin{rep@theorem}}%
 {\end{rep@theorem}}}
\makeatother

\newreptheorem{theorem}{Theorem}
\newreptheorem{lemma}{Lemma}
\newreptheorem{question}{Question}
\newreptheorem{corollary}{Corollary}
\newreptheorem{proposition}{Proposition}

\theoremstyle{definition}
\newtheorem{definition}[theorem]{Definition}
\newtheorem{remark}[theorem]{Remark}

\newtheorem{example}[theorem]{Example}

\newcommand{\Z}{\mathbb{Z}}
\newcommand{\C}{\mathbb{C}}
\newcommand{\N}{\mathbb{N}}
\newcommand{\Q}{\mathbb{Q}}

\newcommand{\id}{\text{id}}

\newcommand{\Aa}{\mathcal A}

\newcommand{\Ll}{\mathcal L}

\newcommand{\Rr}{\mathcal R}

\newcommand{\Tt}{\mathcal T}

\newcommand{\lk}{\ell k}

\makeatletter
\def\@seccntformat#1{%
  \protect\textup{\protect\@secnumfont
    \ifnum\pdfstrcmp{subsection}{#1}=0 \bfseries\fi
    \csname the#1\endcsname
    \protect\@secnumpunct
  }%
}  
\makeatother

\begin{document}

\rhead{\thepage}
\lhead{\author}
\thispagestyle{empty}


\raggedbottom
\pagenumbering{arabic}
\setcounter{section}{0}


\title{Slice disks modulo local knotting}

\author{Jeffrey Meier}
\address{Department of Mathematics, Western Washington University, Bellingham, WA 98229}
\email{jeffrey.meier@wwu.edu}
\urladdr{http://jeffreymeier.org} 

\author{Allison N.  Miller}
\address{Department of Mathematics \& Statistics, Swarthmore College, Swarthmore, PA 19081}
\email{amille11@swarthmore.edu}
\urladdr{https://sites.google.com/view/anmiller/}

\begin{abstract}
The second author and Powell asked whether there exist knots bounding infinitely many slice disks that remain pairwise nonisotopic, even after local knotting. We answer this question in the affirmative, giving many classes of examples distinguished by the kernels of the inclusion-induced maps on the fundamental group.
Along the way, we give a classification of fibered, homotopy-ribbon disks bounded by generalized square knots up to isotopy modulo local knotting, extending work of the first author and Zupan.

We conclude with a discussion of how invertible concordances and satellite operations can produce new examples of knots bounding many inequivalent slice disks. In particular, we give examples of fibered, hyperbolic knots bounding infinitely many fibered, ribbon disks that are pairwise nonisotopic modulo local knotting. The closed monodromies of these knots are pseudo-Anosov mapping classes that have infinitely many distinct handlebody extensions,  
a curiosity that may be of independent interest.

\end{abstract}

\maketitle

\section{Introduction}
\label{sec:introduction}

This paper concerns the classification of disks in $B^4$ bounded by a given slice knot.
This is a challenging problem, whose nature is highly dependent on the notion of equivalence placed on the disks in question.
For example,  the classification of slice disks bounded by the unknot up to isotopy is precisely the classification of knotted spheres in $S^4$ up to isotopy.
In fact,  it is not hard to show that any slice disk for any knot $K$ can be modified by local knotting to give infinitely many nonisotopic slice disks.
One way to circumvent this problem of local knotting is to work with a relation in which two disks are considered equivalent if they become isotopic after connected summation with 2--knots.
Disks that are related in this way are called \emph{isotopic modulo local knotting},  or \emph{LK-isotopic} for short; see Definition~\ref{def:LK}. 
In this setting, the unknot bounds, essentially by definition, a unique slice disk; see Proposition~\ref{prop:unknot}.
The second author and Powell gave examples of knots with arbitrarily many slice disks that are pairwise not  LK-isotopic rel.~boundary~\cite{MilPow_19} (see also~\cite{JuhZem_disks}), and asked the following. 

\begin{question}
\label{quest:infinitelymanydisks?}
	Is there a knot $K$ with infinitely many slice disks that are pairwise nonisotopic rel.  boundary modulo local knotting? 
\end{question}

In this paper, we answer this question in the affirmative.
Our first result classifies twist-spun disks up to LK-isotopy rel. boundary; see Section~\ref{sec:twist-spun} for details.

\begin{reptheorem}{thm:twist-spins}
	Let $J$ be a nontrivial knot. 
	The  $n$--twist-spun ribbon disks  $\{D_n\}_{n \in \Z}$  for $J \# \overline J$ are pairwise nonisotopic rel. boundary modulo local knotting. 
\end{reptheorem}

Note that the $n$--twist-spun disks for $J \# \overline J$ \emph{are} pairwise isotopic via isotopies moving the boundary.  
The bulk of the paper is devoted to a detailed study of the fibered, homotopy-ribbon disks bounded by generalized square knots introduced by the first author and Zupan~\cite{MeiZup_Qpq, MeiZup_disks}, culminating in the following stronger answer to Question~\ref{quest:infinitelymanydisks?}. 

\begin{theorem}
\label{thm:main_answer}
	Let $p,q>0$ be coprime.
	The generalized square knot $Q_{p,q} = T_{p,q}\#\overline T_{p,q}$ bounds infinitely many fibered, homotopy-ribbon disks that are pairwise nonisotopic modulo local knotting. 
\end{theorem}

We remark that when $\{p,q\}=\{2,2k+1\}$,  then infinitely many of the fibered, homotopy-ribbon disks of Theorem~\ref{thm:main_answer} are known to be ribbon~\cite[Theorem 1.6]{MeiZup_disks}.
Theorem~\ref{thm:main_answer} is an immediate consequence of the following two results,  whose proofs rely on the classification theorems for fibered, homotopy-ribbon disks bounded by generalized square knots given in~\cite{MeiZup_Qpq, MeiZup_disks}.
For fixed $p,q$, these disks are parametrized by $c/d\in\Q$ with $c$ even,  and all have diffeomorphic exteriors.

\begin{reptheorem}{thm:main_IRB}
	Two fibered, homotopy-ribbon disks $D_{c/d}$ and $D_{c'/d'}$ bounded by $Q_{p,q}$ are isotopic rel. boundary modulo local knotting if and only if they are isotopic rel. boundary if and only if $c'/d'=c/d$.
\end{reptheorem}

\begin{reptheorem}{thm:main_I}
	Two fibered, homotopy-ribbon disks $D_{c/d}$ and $D_{c'/d'}$  bounded by $Q_{p,q}$ are isotopic modulo local knotting if and only if they are isotopic if and only if there is some $n \in \Z$ such that $c'/d'=\pm(c+2nd)/d$.
\end{reptheorem}

The proofs of Theorems~\ref{thm:main_IRB} and ~\ref{thm:main_I} are geometric and group-theoretic in nature.  We defer the details to Section~\ref{sec:gsks},  noting only that the key tool is the following.

\begin{proposition}\cite[Proposition 6.1]{MilPow_19}\label{prop:kernelob}
Let $D$ and $D'$ be slice disks for $K$.  
If $$\ker \left( \pi_1(S^3 \setminus K) \to \pi_1(B^4 \setminus D) \right) \neq \ker \left( \pi_1(S^3 \setminus K) \to \pi_1(B^4 \setminus D') \right)$$  then $D$ and $D'$ are nonisotopic rel. boundary modulo local knotting. 
\end{proposition}

All the examples described thus far have the form $J\#\overline J$,  so one might wonder if there are other knots with infinitely many inequivalent slice disks.
In Section~\ref{sec:examples}, we show that inequivalence up to isotopy modulo local knotting often persists through various geometric constructions.
First, we use invertible concordances to obtain fibered, hyperbolic knots with infinitely many non LK-isotopic disks. 

\begin{repcorollary}{coro:hyp_fib}
	There exist infinitely many fibered, hyperbolic knots each of which bounds infinitely many fibered, ribbon disks that are pairwise nonisotopic modulo local knotting. 
\end{repcorollary}

The closed monodromies of such knots are pseudo-Anosov and have infinitely many handlebody extensions; this improves work of Long, which shows that the closed monodromy of $\mathbf{10_{123}}$ is pseudo-Anosov and has two handlebody extensions~\cite{Long}.

\begin{repcorollary}{coro:pseudoA}
	There exist infinitely many distinct pseudo-Anosov mapping classes each of which has infinitely many distinct handlebody extensions. 
\end{repcorollary}

We then investigate situations where inequivalence up to isotopy modulo local knotting persists under various satellite operations, yielding the following corollaries.

\begin{repcorollary}{coro:non0winding}
	Let $P= P(U)$ be a pattern with nonzero winding number such that $P(U)$ is slice (respectively, ribbon) and let $J$ be  a nontrivial knot.
	Then $P(J \#\overline J)$ has infinitely many slice disks (respectively, ribbon disks) that are pairwise nonisotopic rel. boundary modulo local knotting. 
\end{repcorollary}

\begin{repcorollary}{coro:trefoil_pattern}
	Let $P$ be a pattern with $P(U)= T_{3,2} \#\overline T_{3,2}$.  
	Then for any slice (respectively, ribbon) knot $K$,  the satellite $P(K)$ has infinitely many slice disks (respectively, ribbon disks) that are pairwise nonisotopic rel. boundary modulo local knotting.  
\end{repcorollary}

The proof of Corollary~\ref{coro:trefoil_pattern} relies on the following Alexander module computation. 

\begin{repproposition}{coro:alternate}
	The ribbon disks $\left\{D_{2/(2k+1)}\right\}_{k\geq0}$ bounded by the square knot are pairwise nonisotopic rel. boundary modulo local knotting, as detected by $\ker( \mathcal{A}(Q) \to \mathcal{A}(D_{2/(2k+1)}))$. 
\end{repproposition}

While the Alexander module method is neither new to this paper (see e.g.~\cite{MilPow_19, Miyazaki86}) nor generally as powerful as the full fundamental group approach,  it is needed for Corollary~\ref{coro:trefoil_pattern} and  is generally far more computationally convenient.
We include detailed computations in Section~\ref{sec:Alex}.  

\subsection{Directions for future work}
We conclude the introduction with a  brief survey of what is known regarding classification of slice disks for a given knot,  together with some open questions. 
For a slice knot $K$,  we let 
\begin{align*}
\mathcal{HR}(K) &= \{\text{homotopy ribbon disks for $K$}\} / \text{ isotopy rel. boundary} \\
\mathcal{D}(K) &= \{\text{slice disks for $K$}\} / \text{ isotopy rel. boundary modulo local knotting}.
\end{align*}

In this language,  Proposition~\ref{prop:unknot} states that $|\mathcal{D}(U)|=1$, while it is open whether $|\mathcal{HR}(U)|=1$ or indeed even whether the unknot has a unique ribbon disk.  Some progress in this direction comes from work of Scharlemann showing that any ribbon disk for the unknot that is  to the standard one must have at least three minima~\cite{Sch_85}.
Additionally,  Gabai's solution to the Property R Conjecture combines with seminal work of Casson-Gordon~\cite{CasGor_83} to imply that the unknot bounds a unique fibered, homotopy-ribbon disk up to isotopy rel. boundary~\cite{Gab_87}. 
If one passes from the smooth category to the topologically locally flat category,  a bit more is known.
Freedman's disk embedding theorem implies that any knot with trivial Alexander polynomial bounds a unique $\Z$--homotopy-ribbon disk up to isotopy rel. boundary~\cite{Fre_84, FreQui_90, DET}.
This work was extended to $BS(1,2)$--homotopy-ribbon disks by Conway and Powell~\cite{ConPow_21}; work of Conway outlines some obstructions to pushing this program forward to other groups~\cite{Con_24}.

There is a natural inclusion-induced map $\iota_K$ from $ \mathcal{HR}(K)$ to $ \mathcal{D}(K)$.  
 From this perspective,  Theorem~\ref{thm:twist-spins} establishes that the image of $\iota_K$ is infinite whenever $K=J \#\overline{J}$ for $J$ nontrivial, and Corollary~\ref{coro:hyp_fib} establishes the same result for infinitely many hyperbolic $K$.  
On the other hand, the following remains open.

\begin{question}
\label{quest:finitedisks?}
	Is there a nontrivial slice knot $K$ for which the image of $\iota_K$  is finite? 
\end{question}

We conjecture that the Stevedore knot $\mathbf{6_1}$ and the knots $\#^k \mathbf{9_{46}}$ are examples of knots for which $D(K)$ is finite and hence for which the answer to Question~\ref{quest:finitedisks?} is `Yes'.

It is also open whether the map $\iota_K$ is injective.
Theorem~\ref{thm:main_IRB} shows that, for a generalized square knot $Q$, the map $\iota_Q$ is injective when restricted to the subset of $\mathcal{HR}(Q)$ consisting of \emph{fibered}, homotopy-ribbon disks.
Rephrasing,  we ask the following. 

\begin{question}
\label{quest:nonisotoiso}
	Are there nonisotopic homotopy-ribbon disks that are isotopic modulo local knotting?
\end{question}
 
We note that the question of whether $\iota_K$ is surjective is a slice-ribbon-type conjecture in the sense that surjectivity would imply that every slice disk is LK-isotopic to a homotopy-ribbon disk.

As usual, one might wonder about the interaction between topological and smooth phenomena.
Akbulut provided examples of exotic smooth slice disks that are topologically isotopic rel. boundary but not smoothly isotopic rel. boundary~\cite{Akb_91}.
Are these disks smoothly isotopic rel. boundary modulo local knotting? 
(See also \cite{GuthEtc} for examples of disks that are topologically isotopic rel. boundary but not smoothly isotopic.) 
Working within the topological category,  one might also ask whether the obstructions of Proposition~\ref{prop:kernelob} are complete.

\begin{question}
Does there exist a topologically slice knot $K$ and slice disks $D$ and $D'$ for $K$ that have 
\[\ker \left( \pi_1(S^3 \setminus K) \to \pi_1(B^4 \setminus D) \right) = \ker \left( \pi_1(S^3 \setminus K) \to \pi_1(B^4 \setminus D') \right)
\] 
and yet which are not topologically  isotopic rel. boundary modulo local knotting?
\end{question}

This question is resolved under certain additional assumptions by \cite[Corollary 1.8]{Con_24}.

Finally, let $d$ be the generalized stabilization distance of~\cite{MilPow_19},  where disks $D$ and $D'$ have distance no more than $k$ if and only if $D$ and $D'$ are LK-isotopic rel. boundary after $k$ 1--handle additions to each.  
The fibered, homotopy-ribbon disks of Theorems~\ref{thm:main_IRB} and \ref{thm:main_I} are obtained by surgering a fixed minimal-genus Seifert surface for $Q_{p,q}$,  and hence have pairwise generalized stabilization distance at most $g(Q_{p,q}) = (p-1)(q-1)$.
The following therefore remains open. 

\begin{question}
	Does there exist a knot $K$ with slice disks $\{D_n\}_{n\in\N}$ such that $\{d(D_n, D_m)\}_{n,m\in\N}$ is unbounded?
\end{question}

We note that it is unclear what lower bounds could apply to answer this question in the affirmative.  The methods of~\cite{MilPow_19} rely upon the generating rank of certain submodules of the Alexander module and twisted Alexander modules,  but these are always bounded above by the (finite) generating rank of the Alexander module of the knot itself.

\subsection*{Acknowledgements}
 
The first author is grateful to Isaac Sundberg for bringing the work of the second author and Powell to his attention and to Alex Zupan for discussions that led to some of the ideas in Section~\ref{sec:gsks}, and he would like to thank the Max Planck Institute for Mathematics for providing an ideal environment for collaboration in the fall of 2022, during which time the aforementioned discussions took place.

The first author was supported by the NSF grant DMS-2405324.  The second author was supported by an AMS-Simons Research Enhancement Grant for PUI Faculty. 

\section{Definitions}
\label{sec:notions}

In this section, we establish notational conventions, give formal definitions, and explore basic properties.

\subsection{Notation}
\label{subsec:notation}

Throughout, our slice disks are smooth, but our obstructions to being isotopic modulo local knotting hold in the topological category.
We let $\nu(N)$ denote an open tubular neighborhood of a submanifold $N\subset M$.
Submanifolds will always be \emph{neatly embedded}: if $N\subset M$, then $\partial N\subset\partial M$ and $N\pitchfork\partial M$.
For a codimension-two submanifold $N\subset M$, we let $\pi(N)$ denote $\pi_1(M\setminus\nu(N))$.
In this case, an element of $\pi(N)$ is called \emph{meridional} if it is freely homotopic in $M\setminus\nu(N)$ to a curve that bounds an embedded disk in $M$ that intersect $N$ in a single point.

Given a knot $J$, we let $\overline J$ denote the mirror reverse of $J$.
More generally, $\overline{(M,N)}$ denotes $(M,N)$ with both orientations reversed, so $\overline{(S^3,J)} = \left(S^3,\overline J\right)$.

\subsection{Definitions}
\label{subsec:definitions}

Let $K$ be a knot in $S^3$.
A neatly embedded disk $D\subset B^4$ with $\partial D = K$ is a called a \emph{slice disk for $K$}.
If the inclusion-induced map $\pi(K)\to\pi(D)$ is a surjection, then $D$ is called \emph{homotopy-ribbon}.
If $D$ can be isotoped to have no local maxima with respect to the standard Morse function on $B^4$, then $D$ is called \emph{ribbon}.
The knot $K$ is called \emph{slice}, \emph{homotopy-ribbon}, or \emph{ribbon}, respectively.
The following sequence of set-inclusions is immediate:
$$\{\text{ribbon knots}\}\subseteq\{\text{homotopy-ribbon knots}\}\subseteq\{\text{slice knots}\}.$$
However, it is an open question (which filters the Slice Ribbon Conjecture~\cite{Fox_62}) whether either of the above set inclusions is strict.
Many of the slice disks encountered in this paper are fibered, homotopy-ribbon disks, which have a rich history of study~\cite{CasGor_83, LarMei_15, MeiZup_Qpq, Mil_21, Miy_94}. 
In fact, the proof of our main result, Theorem~\ref{thm:main_IRB}, depends implicitly on the property of fiberedness.

Two slice disks $D$ and $D'$ for $K$ are \emph{isotopic} (respectively, \emph{isotopic rel. boundary}) if there is an ambient isotopy $\Phi_t$ of $B^4$ with $\Phi_1(D)=D'$ (respectively, with $\Phi_1(D) = D'$ and $\Phi_t\vert_{S^3} = \id_{S^3}$).
A \emph{2--knot} is a smoothly embedded 2--sphere in $S^4$.
We now introduce the main definition.

\begin{definition}
\label{def:LK}
	Two slice disks $D$ and $D'$ for $K$ are \emph{isotopic modulo local knotting}, or \emph{LK-isotopic}, (respectively, \emph{isotopic rel. boundary modulo local knotting}, or \emph{LK-isotopic rel. boundary}) if there exist 2--knots $S$ and $S'$ such that $D\#S$ and $D\#S'$ are isotopic (respectively, isotopic rel. boundary).   
\end{definition}

The following well-known result is a quick consequence of these definitions. 
\begin{proposition}
\label{prop:unknot}
	Any two slice disks for the unknot are LK-isotopic rel. boundary.
\end{proposition}

\begin{proof}
	Let $D$ and $D'$ be slice disks for the unknot $U$.
	Let $F$ denote the slice disk for $U$ that can be isotoped rel. boundary to lie in $S^3$.
	(There is a unique such slice disk up to isotopy rel. boundary~\cite{Liv_82}.)
	Consider the 2--knots $(S^4, S) = \overline{(B^4,F)}\cup_{(S^3,U)}(B^4,D)$ and $(S^4, S') = \overline{(B^4,F)}\cup_{(S^3,U)}(B^4,D')$.
	Then $D\#S'$ and $D'\#S$ are both punctured copies of  $S\#S'$.
	Since the map from 2--knots to slice disks for the unknot given by puncturing is a bijection, $D\#S$ and $D'\#S'$ are isotopic rel. boundary.  So $D$ and $D'$ are LK-isotopic rel. boundary.
\end{proof}

\section{Twist-spun disks}
\label{sec:twist-spun}

Let $J$ be a nontrivial knot.
Let $(B^3,J^\circ) = (S^3,J)^\circ$ denote the result of puncturing $(S^3,J)$ by removing a small, open ball-neighborhood of a point on $J$, so $J^\circ$ is a knotted arc in $B^3$.
Let $(B^4,D_0) = (B^3,J^\circ)\times [0,2\pi] $, and note that $\partial (B^4, D_0) = (S^3,J\#\overline J)$.
We call $D_0$ the \emph{product ribbon disk} for $J\#\overline J$.

Let $A$ be a vertical axis for $B^3$ intersecting $\partial B^3$ in $\partial J^\circ$, and for $t \in [0,1]$ let $\tau_t\colon B^3\to B^3$ denote rotation through $2\pi t$ radians about $A$. 
Define $\Tt\colon B^3\times I\to B^3\times I$ by $\Tt(x,t) = (\tau_t(x),t)$.

\begin{definition} \label{defn:twistspundisk}
	For $n\in\Z$, let $(B^4, D_n) = \Tt^n(B^4, D_0)$.
	Informally,    $D_n$ is traced out by rotating $J^\circ$ a total of $n$ times as we flow $(B^3,J^\circ)$ through a interval of time.
	We call $D_n$ the \emph{$n$--twist spun disk} for $J\#\overline J$.
\end{definition}

The $D_n$ represent an infinite family of isotopic ribbon disks for $K = J\#\overline J$.
Let
$$\iota_n\colon \pi(K)\to\pi(D_n)$$
be the (surjective) homomorphism induced by inclusion.
We are now ready to prove our first result.

\begin{theorem}\label{thm:twist-spins}
	Let $J$ be a nontrivial knot. 
	The  $n$--twist-spun ribbon disks  $\{D_n\}_{n \in \Z}$  for $J \# \overline J$ are pairwise nonisotopic  rel. boundary modulo local knotting.
\end{theorem}

\begin{proof}
	Fix $n,m\in\Z$. By Proposition~\ref{prop:kernelob},  it suffices to show that  $\ker(\iota_m) \neq \ker(\iota_n)$ whenever $m \neq n$. 
	Consider the 2--knot $(S^4, S)$ defined as
	$$(S^4,S) = (B^4,D_n)\cup_{(S^3,K)}\overline{(B^4,D_m)}.$$
	
	The isotopy $\tau_t$ described above gives rise to a diffeomorphism $\tau\colon (S^3,K)\to (S^3,K)$ given by rotating one summand of $K = J\#\overline J$ and fixing the other; see Section~3 of~\cite{MeiZup_Qpq} for explicit details.
Since $\tau$ is isotopic to the identity, we can apply $\tau^{-m}$ to $(S^3,K)\subset (S^4,S)$, tapering it to the identity in a collar neighborhood of $(S^3,K)$.
	The effect is that $(S^4,S)$ is diffeomorphic to
	$$(S^4, S_{n-m}(J)) = (B^4, D_{n-m})\cup_{(S^3,K)}\overline{(B^4,D_0)},$$
	which is the $(n-m)$--twist spin of $(S^3,J)$ defined by Zeeman~\cite{Zee_65}.
	
	By Seifert-Van Kampen, we have that
	$$\pi(S)= \pi(D_m) * \pi(D_n)/ \langle \iota_n(g)\iota_m(g^{-1})	\rangle_{g \in \pi(K)}.$$
	Since $\pi(D_m)= \pi(K)/ \ker(\iota_m)$ and $\pi(D_n)= \pi(K)/ \ker(\iota_n)$, we have 
	$$\pi(S)= \pi(K)/ \langle \ker(\iota_m),  \ker(\iota_n) \rangle = \pi(D_m)/ \ker(\iota_n) = \pi(D_n)/ \ker(\iota_m).$$
	
	If we suppose $\ker(\iota_m) = \ker(\iota_n)$, then we find that 
	\begin{align} \label{eqn:pi1isos} \pi(S_{n-m}) \cong \pi(S) \cong \pi(D_n) \cong\pi(J).
	\end{align}

Let $\mu \in \pi(J)$ be a meridional element.
Since the isomorphisms of Equation~\ref{eqn:pi1isos} are induced by inclusion, they send meridional elements to meridional elements.
Therefore $\mu$ is meridional in $\pi(S_{n-m})$,  and hence by~\cite[Section~2.1.2]{CarKamSai},  we have that $\mu^{n-m}$ is central in $\pi(S_{n-m}) \cong \pi(J)$. 

	However, since $\pi(J)$ is the group of a nontrivial knot, its center is trivial unless $J$ is a torus knot $T_{p,q}$~\cite[Chapter~6]{Kaw_96}. 
	Moreover, the center of $\pi(T_{p,q})$ is generated by an element abelianizing to $pq\in\Z$~\cite[Chapter~6]{Kaw_96}.
	But no nontrivial elements of the center can be meridional.
	It follows that $n-m=0$, as desired.
\end{proof}

\section{Disks for generalized square knots}
\label{sec:gsks}

Fix coprime integers $1 < q < p$, and let $T_{p,q}$ denote the $(p,q)$--torus knot.
We refer to the knot $Q = Q_{p,q} = T_{p,q}\#\overline T_{p,q}$ as a \emph{generalized square knot}.
In~\cite{MeiZup_Qpq}, the first author and Zupan introduced an infinite family of fibered, homotopy-ribbon disks bounded by $Q$,  indexed by the rational numbers with even numerator:
$$\Rr = \{D_{c/d} \mid c/d\in\Q,\text{ $c$ even}\}.$$
It was shown that every fibered, homotopy-ribbon disk for $Q$ is in $\Rr$; the members of $\Rr$ are pairwise nonisotopic rel. boundary; and the members of $\Rr$ have diffeomorphic exteriors~\cite[Theorem~1.6]{MeiZup_Qpq}.
Note that for $n\in\Z$, the disk $D_{2n/1}$ is the $n$--twist spun ribbon disk for $Q$, which was introduced in more generality in Section~\ref{sec:twist-spun}.
In~\cite{MeiZup_disks}, it was shown that $D_{c/d}$ and $D_{c'/d'}$ are isotopic (without restriction on the boundary) if and only if there is $n \in \Z$ such that
$$\frac{c'}{d'} = \pm \frac{c+2nd}{d}.$$

The goal of this section is to give the proof of the results of this paper that classify the disks of $\Rr$ up to LK-isotopy rel. boundary and LK-isotopy, respectively.

\begin{theorem}
\label{thm:main_IRB}
	Two fibered, homotopy-ribbon disks $D_{c/d}$ and $D_{c'/d'}$ bounded by $Q_{p,q}$ are isotopic rel. boundary modulo local knotting if and only if they are isotopic rel. boundary if and only if $c'/d'=c/d$.
\end{theorem}

As in~\cite{MeiZup_disks}, this theorem can be strengthened to give a classification of fibered, homotopy-ribbon disks bounded by $Q$ up to isotopy, with no restriction on the boundary.
The proof of the following theorem as a consequence of Theorem~\ref{thm:main_IRB} is philosophically identical to the proof of~\cite[Theorem~1.5]{MeiZup_disks}, but we include it here for completeness and the reader's convenience.

\begin{theorem}
\label{thm:main_I}
	Two fibered, homotopy-ribbon disks $D_{c/d}$ and $D_{c'/d'}$  bounded by $Q_{p,q}$ are isotopic modulo local knotting if and only if they are isotopic if and only if there is some $n \in \Z$ such that $c'/d'=\pm(c+2nd)/d$.
\end{theorem}

\begin{proof}[Proof of Theorem~\ref{thm:main_I},  assuming Theorem~\ref{thm:main_IRB}]
	If two disks are isotopic, then they are (trivially) LK-isotopic. 
	For the converse, suppose $D_{c/d}$ and $D_{c'/d'}$ are LK-isotopic.  
	This means that there exist 2--knots $(S^4, S)$ and $(S^4,S')$ such that the disks $(B^4,D) = (B^4,D_{c/d})\#(S^4,S)$ and $(B^4, D') = (B^4,D_{c'/d'})\#(S^4, S')$ are isotopic.
	By~\cite[Proposition~2.1]{MeiZup_disks} (see also~\cite[Lemma~2.2]{JuhZem_disks}), this is equivalent to the existence of a diffeomorphism $\Psi\colon B^4\to B^4$ taking $D'$ to $D$.

	Let $\psi = \Psi\vert_{S^3}$.
	Since $\psi(Q) = Q$, we can regard $\psi$ as a symmetry of $Q$.
	In~\cite[Proposition~3.3]{MeiZup_disks}, automorphisms $\alpha$, $\beta$, and $\tau$ representing generators of the mapping class group $\textsc{Sym}(S^3,Q)$ were identified.
	It follows that there is an ambient isotopy $f_t\colon S^3\to S^3$ such that $f_t(Q)=Q$ for all $t$ and such that $\psi\circ f_1$ is equal to a product of $\alpha$, $\beta$, and $\tau$.
	We can extend $f_t$ to an ambient isotopy $F_t\colon B^4\to B^4$ that is supported in a collar neighborhood of $S^3= \partial B^4$ and satisfies $F_t(D') = D'$ for all $t$.
	It follows that $\Psi\circ F_1$ is a diffeomorphism of $B^4$ taking $D'$ to $D$ whose restriction of $S^3$ is a product of $\alpha$, $\beta$, and $\tau$.
	The upshot of all of this is that we might as well assume without loss of generality that the original restriction $\Psi\vert_{S^3} = \psi$ is itself a product of $\alpha$, $\beta$, and $\tau$.
	
	Let $\overline\alpha$, $\overline\beta$, and $\overline\tau$ denote the extensions of the generators across $B^4$ defined in~\cite[Proposition~4.1]{MeiZup_disks}, and let $\overline\psi$ denote the corresponding extension of $\psi$.
	The effects of these extensions on the disk $D_{c/d}$ are described in~\cite[Proposition~4.1]{MeiZup_disks}.
	In particular, there is some $n\in\Z$ such that $\overline\psi^{-1}(D_{c/d})$ is isotopic rel. boundary to $D_{\pm(c-2nd)/d}$.	
	It follows that $\overline\psi^{-1}(D)$ is isotopic rel. boundary to $D_{\pm(c-2nd)/d}\#\overline\psi^{-1}(S)$.
	(Here, it is helpful to think of $D = D_{c/d}\#S$ as differing from $D_{c/d}$ only within a small 4--ball where $D_{c/d}$ is locally knotted according to $S$.)
		
	Let $\Psi' = \overline\psi^{-1}\circ\Psi$, and note that $\Psi'$ restricts to the identity on $\partial B^4$.
	Then we have
	$$\Psi'(D') = \overline\psi^{-1}\circ\Psi(D') = \overline\psi^{-1}(D) \stackrel{\partial}{\simeq} D_{\pm(c-2nd)/d}\#S,$$
	where the final equivalence is isotopy rel. boundary.
	Since $\Psi'$ is a diffeomorphism rel. boundary, and since $D' = D_{c'/d'}\#S'$, the conclusion is that $D_{c'/d'}$ and $D_{\pm(c-2nd)/d}$ are LK-isotopic rel. boundary,   again by~\cite[Proposition~2.1]{MeiZup_disks} or~\cite[Lemma~2.2]{JuhZem_disks}.
	By Theorem~\ref{thm:main_IRB}, it follows that $D_{c'/d'}$ and  $D_{\pm(c-2nd)/d}$ are isotopic rel. boundary,  and by~\cite[Theorem~1.4]{MeiZup_disks}, this implies that $d'=d$ and $c'= \pm (c-2nd)/d$.
	Therefore, $D_{c/d}$ and $D_{c'/d'}$ are isotopic. 
\end{proof}

To prove Theorem~\ref{thm:main_IRB}, it suffices to show that kernels of the inclusion-induced epimorphisms
$$\iota_{c/d}\colon \pi(Q)\to\pi(D_{c/d})$$
are pairwise distinct in $\pi(Q)$, since this implies the disks are LK-nonisotopic rel.\ boundary,  by~\cite[Proposition~6.1]{MilPow_19}.
This is accomplished in the next result.

\begin{proposition}
\label{thm:kernels}
	$\ker(\iota_{c/d}) = \ker(\iota_{c'/d'})$ if and only if $c/d = c'/d'$.
\end{proposition}

The proof of Proposition~\ref{thm:kernels} relies on Propositions~\ref{prop:generator} and~\ref{prop:nontrivial_class}.
In order to establish these propositions, we must carry out a detailed examination of the exterior $Z_{c/d} = B^4\setminus\nu(D_{c/d})$, paying special attention to a circle-action it inherits from the standard circle-action on $(S^3,T_{p,q})$.
This analysis subsumes the remainder of this section, which culminates with the proof of Proposition~\ref{thm:kernels}, hence Theorem~\ref{thm:main_IRB}.

\subsection{Seifert fibered structures via circle actions}
\ 

We begin by recalling how the torus knot $T_{p,q}$ can be defined by considering a circle-action on $S^3$, and describe the corresponding Seifert fibered structure on the knot exterior.
This circle-action also gives rise to a 4--dimensional Seifert fibered structure on the exterior $Z_0$ of the product ribbon disk $D_0$ bounded by the generalized square knot $Q:=Q_{p,q}=T_{p,q} \#- T_{p,q}$.
In Subsection~\ref{subsec:disk_group}, we use this structure to derive two useful presentations for the group $\pi(D_0)\cong\pi(T_{p,q})$.

For our purposes, a manifold will be called \emph{Seifert fibered} if it admits a smooth, fixed-point-free circle-action. 
All the circle-actions considered here will be fixed-point free; hence, all manifolds admitting circle-actions will be Seifert fibered. 
For an introduction to basic notions of circle-actions on 3--manifolds and 4--manifolds, see~\cite{Fin_77,Fin_78,Orl_72}. 
We will refer to the orbits of points as \emph{circle-fibers}.
Recall that a subset $X$ of a space equipped with an $S^1$--action is \emph{invariant} if $g\cdot X = X$ for all $g\in S^1$
and \emph{equivariant} if $(g\cdot X)\cap X  \in \{ \varnothing, X \}$ for all $g\in S^1$.
Note that a subset $X$ of a Seifert fibered space is invariant exactly when it is a union of circle-fibers and is equivariant provided it intersects every circle-fiber transversely in at most one point (though this is not a necessary condition).

Let $D^2\subset\C$ denote the unit disk, and let $S^3 = \partial(D^2\times D^2)$. Adopting complex polar coordinates, we have
$$S^3 = \left\{\left(r_1e^{i\theta_1},r_2e^{i\theta_2}\right) \mid (r_1=1 \text {, } r_2 \leq 1)  \text{ or } ( r_1 \leq 1 \text{, } r_2=1)\right\}.$$
Let $p>q>1$ be coprime integers, and consider the circle-action on $S^3$ given by
$$\psi\cdot\left(r_1e^{i\theta_1},r_2e^{i\theta_2}\right):= \left(r_1e^{i\left(\theta_1+\frac{2\pi}{p}\psi\right)},r_2e^{i\left(\theta_2+\frac{2\pi}{q}\psi\right)}\right).$$
This action is fixed-point free.  

This action respects the standard genus-one Heegaard splitting $(\Sigma; H_1, H_2)$ of $S^3$, where
\begin{enumerate}
	\item $\Sigma=\left\{\left(e^{i \theta_1}, e^{i \theta_2}\right)\right\}$,
	\item $H_1 = \left\{\left(r_1e^{i\theta_1},r_2e^{i\theta_2}\right)\mid r_1 \leq 1 \text{, } r_2 = 1\right\}$, and
	\item $H_2 = \left\{\left(r_1e^{i\theta_1},r_2e^{i\theta_2}\right) \mid r_1 = 1 \text{, } r_2 \leq 1 \right\}$.
\end{enumerate}
The orbit of a point on $\Sigma$, say $(1,1)$, is the torus knot $J = T_{p,q}$.
For $i=1,2$, the core of $H_i$ is the orbit $c_i$ of a point satisfying $r_i=0$.
The cores are \emph{exceptional fibers}, while the rest of the orbits are \emph{ordinary fibers}.
The Heegaard torus $\Sigma$ is shown in Figure~\ref{fig:circle-action}, where the solid torus $H_1$ is ``inside" the depicted torus and $H_2$ includes the point at infinity. 

We let $\mu_1 = \{\theta_2=0\}, \mu_2 = \{\theta_1=0\} \subset \Sigma$, so $([\mu_1],[\mu_2])$ is an oriented basis for the homology of $\Sigma$ and $\mu_i$ bounds a compressing disk in $H_i$.
In these coordinates, we have
$$[J] = q [\mu_1] +p [\mu_2].$$
Figure~\ref{fig:circle-action} shows this configuration when $(p,q)=(7,5)$.
Let $\eta$ denote a simple closed curve on $\Sigma$ with $[\eta]= a[\mu_1]-b [\mu_2]$ in $H_1(\Sigma)$,  where $a$ and $b$ satisfy $ap+bq=1$ and $0<a<q$,  and which intersects $J$ in a single point.
Note that $\eta$ can be chosen to be equivariant, since it intersects $J$ in a single point.
For example, when $(p,q) = (7,5)$, we have $(a,b) = (3,-4)$.
In general, $\eta$ is a $(-b,a)$--curve.

Observe that $\eta$ is oriented so that $([\eta],[J])$ is an oriented basis for $H_1(\Sigma)$.  We have the following relationship between the two bases we've encountered:
$$\begin{bmatrix} [\eta] \\ [J]\end{bmatrix} = 
\begin{bmatrix} a & -b \\ q & p\end{bmatrix}
\begin{bmatrix} [\mu_1] \\ [\mu_2]\end{bmatrix}.$$
Inverting, we find that 
\begin{align}\label{eqn:musbyetaJ}
[\mu_1] = p[\eta] + b[J] \text{ and } [\mu_2] = -q[\eta] + a[J], 
\end{align}
which will be useful in Subsection~\ref{subsec:tk_group}.

\begin{figure}[h!]
	\centering
	\includegraphics[width=.7\linewidth]{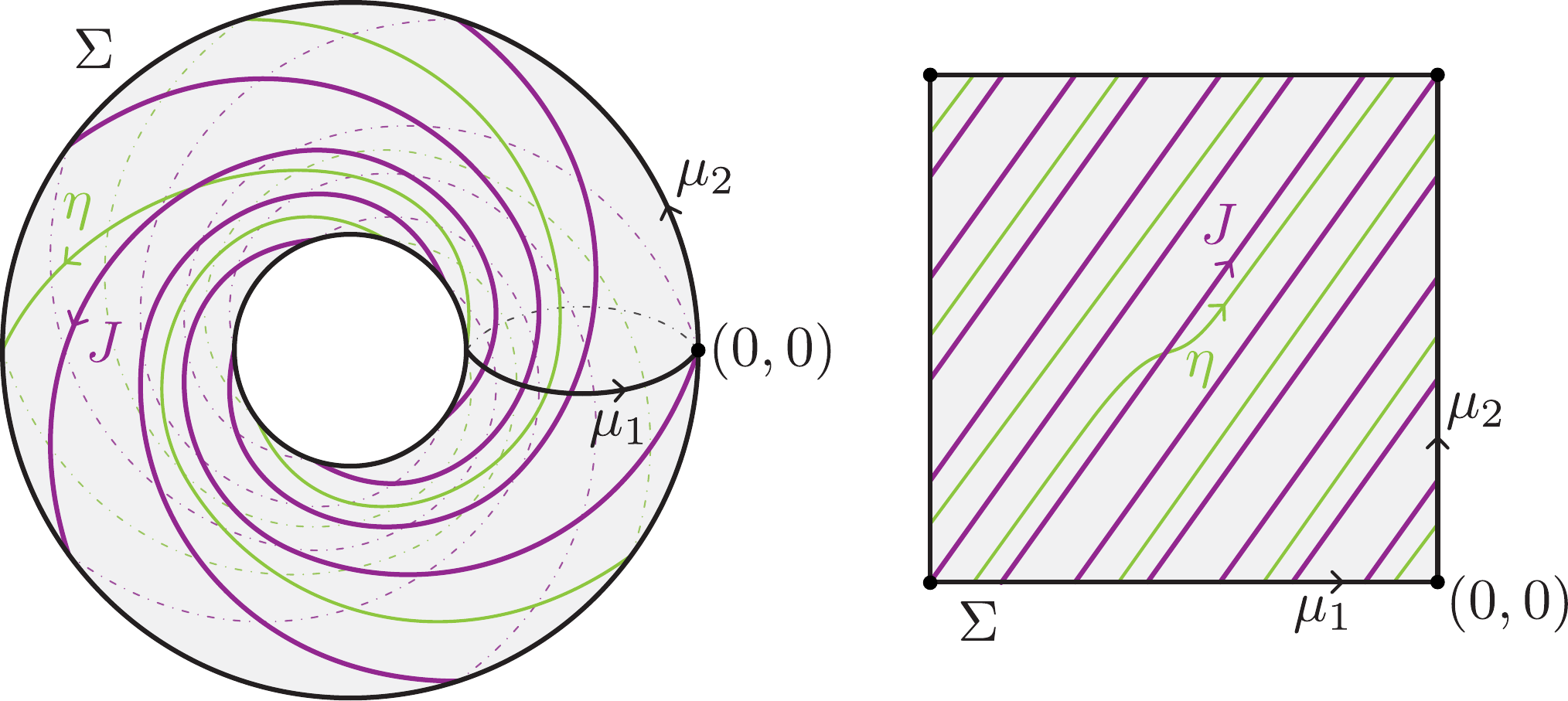}
	\caption{The torus knot $J = T_{7,5}$ as an orbit of a circle-action on $S^3$. Here, the curve $\eta$ is a $(4,3)$--curve.}
	\label{fig:circle-action}
\end{figure}

Let $\nu(c_i)$ be an open, invariant, tubular neighborhood of $c_i$ for $i=1,2$.
The circle-bundle structure on the closure $\overline\nu(c_1)$ is obtained by taking the $I$--bundle structure on $D^2\times I$ and gluing the ends together with a twist through $\frac{2\pi q}{p}$ radians.
For this reason, we refer to $c_1$ as an exceptional fiber of \emph{multiplicity} $p/q$.
Similarly, we see that the circle-bundle structure on $\overline\nu(c_2)$ is obtained by taking the $I$--bundle structure on $D^2\times I$ and gluing the ends together with a twist through $\frac{2\pi p}{q}$ radians.
For this reason, we refer to $c_2$ as an exceptional fiber of \emph{multiplicity} $q/p$.

Let $\nu(J)$ denote an open, invariant, tubular neighborhood of $J$.
The circle-bundle structure on the closure $\overline\nu(J)$ is obtained by taking the $I$--bundle structure on $D^2\times I$ and gluing the ends together with a twist through $2\pi pq$ radians.  In particular,  if we let $c_0\subset\partial \overline\nu(J)\cap\Sigma$ be an ordinary fiber, we have $\lk(c_0,J) = pq$.
Let $\mu_J \subset \partial \overline\nu(J)$ be a meridional curve such that $([\mu_J],[c_0])$ is an oriented basis for $H_1(\partial \overline\nu(J))$.
Let $E=S^3\setminus\nu(J)$ denote the torus knot exterior, which has been chosen to be invariant.

In what follows, if $X$ is any subset of $S^3$, we let $X^*$ denote the orbit space of the circle-action on $X$.
This is most natural when $X$ is invariant,  but we will also use this notation when $X$ is equivariant.
The $H_i^*$ are disks that meet along their common boundary circle $\Sigma^*$, so $\left(S^3\right)^*$ is a 2--sphere.
The points $c_1^*$ and $c_2^*$ are cone points of order $p$ and $q$, respectively.
It is in this sense that we regard $S^3$ as a Seifert fibered space over the orbifold $S^2(p, q)$.
Since $\overline\nu(J)^*$ is a disk, it follows that $S^3\setminus\nu(J)$ is a Seifert fibered space over the orbifold $D^2(p, q)$.
The quotient orbifolds $S^2(p, q)$ and $D^2(p, q)$ and relevant subspaces are shown in Figures~\ref{fig:sphere} and~\ref{fig:disk}.

Regard $S^3\setminus\nu(c_1\cup c_2)$ as $\Sigma\times [1,2]$, and consider the annulus $A=\eta\times [1,2]$.
Write $\partial A = \eta_1\cup\eta_2$, with $\eta_i\subset\partial\overline\nu(c_i)$, so $[\eta_1]=[\eta]=-[\eta_2] \in H_1(A)$.
Since each $\eta\times\{t\}$ intersects a regular fiber in a single point, it follows that $A^*$ is an annulus on $S^2(p, q)$.
Since $A\cap\overline\nu(J)$ is an equivariant meridional disk for $\overline\nu(J)$, we have that $P = A\setminus\nu(J)$ is a pair of pants, as is $P^* = E^*\setminus(\nu(c_1)^*\cup\nu(c_2)^*)$ in $S^2(p, q)$.
We can assume $\mu_J \subset\partial P$, so $\mu_J^* = \partial D^2(p, q)$.

\begin{figure}[h!]
	\begin{subfigure}{.33\textwidth}
		\centering
		\includegraphics[width=.8\linewidth]{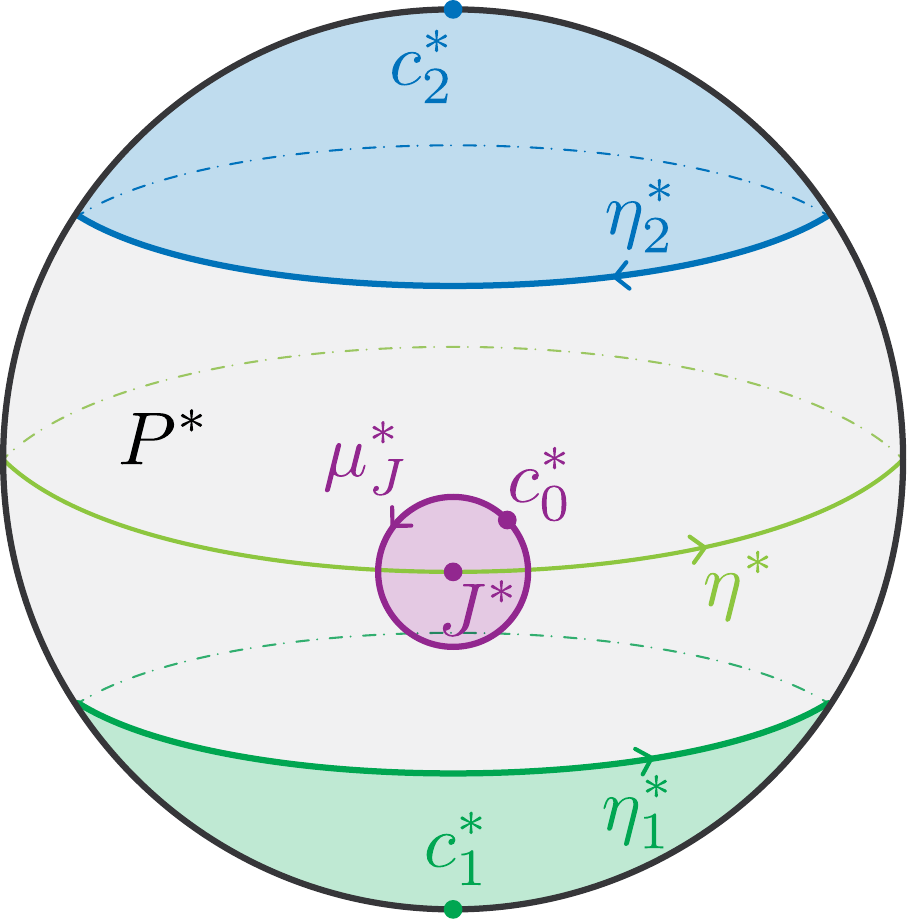}
		\caption{$S^2(p,q)=(S^3)^*$}
		\label{fig:sphere}
	\end{subfigure}%
	\begin{subfigure}{.33\textwidth}
		\centering
		\includegraphics[width=.85\linewidth]{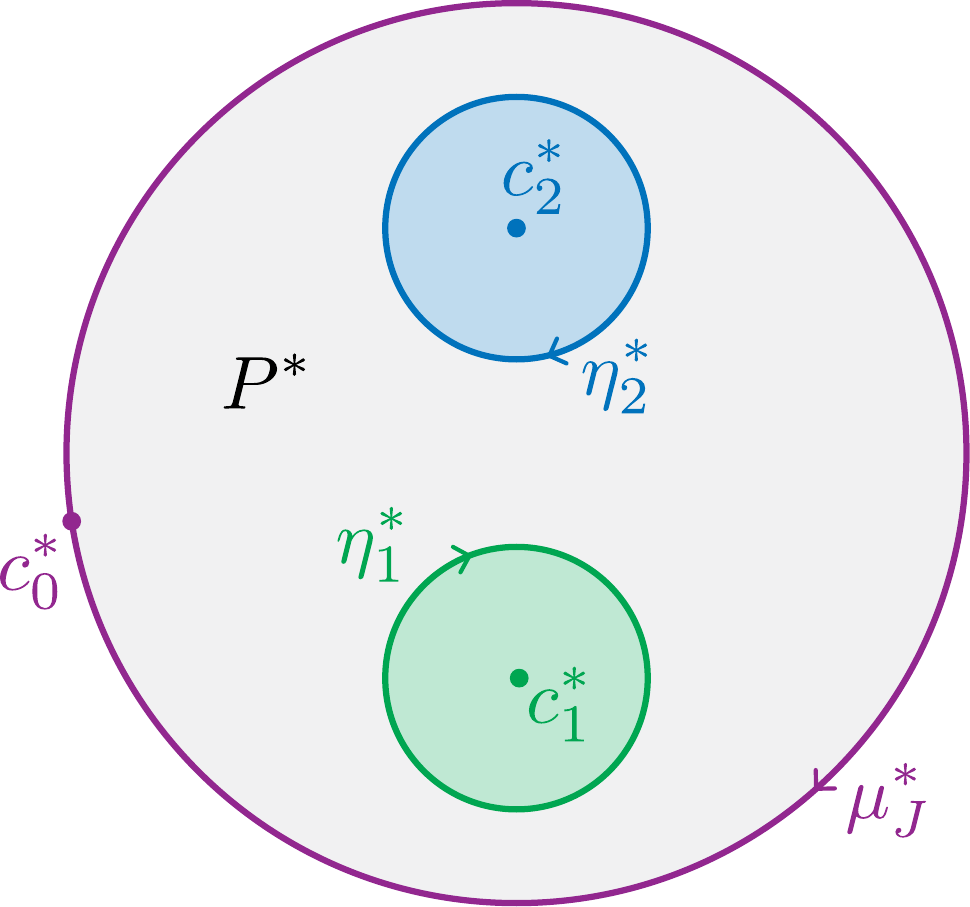}
		\caption{$D^2(p,q)=E^*$}
		\label{fig:disk}
	\end{subfigure}%
	\begin{subfigure}{.33\textwidth}
		\centering
		\includegraphics[width=.8\linewidth]{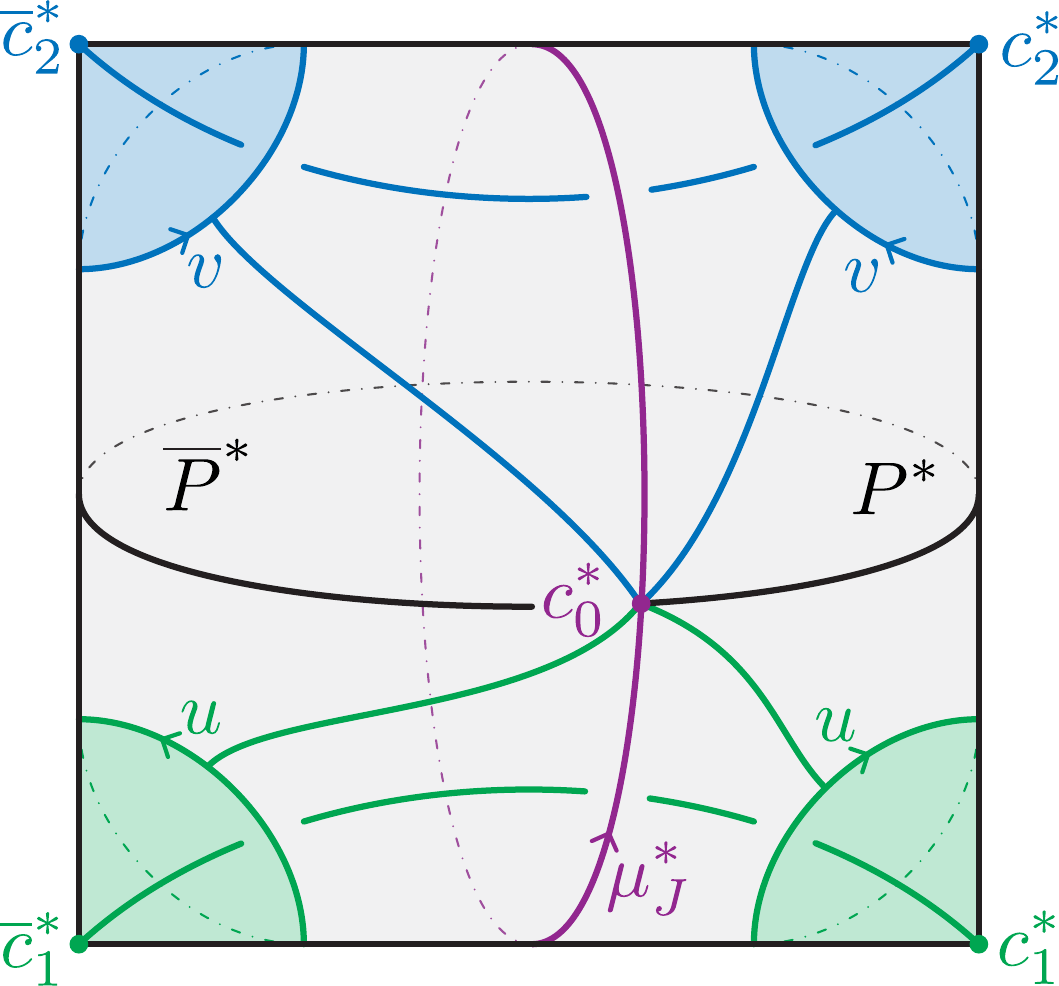}
		\caption{$B^3(p,q)= D^2(p,q) \times I$}
		\label{fig:pillow}
	\end{subfigure}%
	\caption{The relevant quotient orbifolds}
	\label{fig:orbifolds}
\end{figure}

Observe that the circle action restricted to $S^3 \setminus \nu(c_1 \cup c_2)$ gives it the structure of a (trivial) circle bundle over $A^*$,  which further restricts to give $E' = E \setminus \nu(c_1 \cup c_2)$ the structure of a (trivial) circle-bundle over $P^*$. 
In this sense,  it is reasonable to regard $A$ as a section of $S^3 \setminus \nu (c_1 \cup c_2) \to A^*$ and $P$ as a section of $E \setminus \nu(c_1 \cup c_2) \to P^*$.
In fact, in the following subsections, we will \emph{identify} various quotient spaces $X^*$ with preferred sections $X$ in order to relate their fundamental groups.

\subsection{Presentations for the group of the torus knot}
\label{subsec:tk_group}
\ 

We now compute the fundamental group of $E = S^3\setminus\nu(J)$ using the Seifert fibered structure.
Pick as a base-point the point $P\cap c_0$ in $E$.  This point can be identified with its quotient $c_0^*$, since $P$ is our preferred section of $P^*$.
Regard the curves $\eta_i$ as based loops by choosing arcs from $c_0^*$ to $\eta_i^*$ in $P^*$ and lifting to obtain arcs from $c_0\cap P$ to $\eta_i$ in $P$.
Denote the resulting homotopy classes by $u=[\eta_1]$ and $v=[\eta_2]$.
Also,  let
 $h$ denote the homotopy class of $c_0$,  based at $c_0^*=P \cap c_0.$
Recall that $E = (P\times S^1)\cup \overline\nu(c_1)\cup \overline\nu(c_2)$.
We obtain the presentation 
$$\pi_1(P\times S^1,c_0^*) = \left\langle u, v, h \mid uh=hu, vh=hv\right\rangle.$$
The solid tori $\overline\nu(c_i)$ give two more relations.
To identify them, note that
$$[\mu_1] = b [c_0] + p [\eta_1] \in H_1(\partial \overline\nu(c_1)),$$
and
$$[\mu_2] = a [c_0] + q  [\eta_2] \in H_1(\partial \overline\nu(c_2)).$$
Here, we are using the $\Sigma\times[1,2]$ product structure to identify $\partial\overline\nu(c_i)$ with $\Sigma$ and regard $\mu_i$ and $c_0$ as curves on $\partial\overline\nu(c_i)$, as well as using Equation~\ref{eqn:musbyetaJ}. 

Thus, we obtain the presentation
\begin{align}\label{eqn:pres1pi1E}
\pi_1(E,c_0^*) = \left\langle u, v, h \mid u^p = h^{-b}, v^q = h^{-a}, uh = hu, vh = hv\right\rangle.
\end{align}
We note in passing that, with the chosen orientations, in $\pi_1(E,c_0^*)$, we have
$$[\mu_J] = [\eta_2][\eta_1] = vu.$$
This relation is induced by the pair of pants $P$ itself.

We now make a change of variable to recover the standard presentation for $\pi(J)$.
Regarding $c_1$ and $c_2$ as loops based at $P\cap c_0$, we see that $\eta_1$ is homotopic to $c_1^{-b}$ in $H_1$, $\eta_2$ is homotopic to $c_2^{-a}$ in $H_2$, and $c_0$ is homotopic to both $c_1^p$ and $c_2^q$.
Let $x = [c_1]$ and $y = [c_2]$.
Since $u = x^{-b}$, $v = y^{-a}$, and $h = x^p = y^q$, the first two relations of Presentation~\ref{eqn:pres1pi1E} become vacuous and the second two relations reduce to a single relation.  We therefore obtain
$$\pi_1(E,c_0^*) = \left\langle x, y \mid x^p = y^q\right\rangle,$$
with $[\mu_J] =y^ax^b.$

\begin{remark}
\label{rmk:identify_1}
	The utility of the above set-up is that the pairs of pants $P$ and $P^*$ can be identified.
	In this way, we could have defined $u$ and $v$ to be $[\eta_1^*]$ and $[\eta_2^*]$ in $\pi_1(P^*,c_0^*)$.
	In Subsection~\ref{subsec:bridge_relator}, we will lean more heavily on this perspective.
\end{remark}

\subsection{Presentations for the group of the product ribbon disk}
\label{subsec:disk_group}
\ 

Consider the product ribbon disk $(B^4,D_0) =(S^3,J)^\circ\times I$, which was discussed in Section~\ref{sec:twist-spun}.
Then, $\partial(B^4,D_0) = (S^3,Q)$, where $Q = Q_{p,q} = \overline J\# J$.
Let $Z_0 = B^4\setminus\nu(D_0)$, and note that $Z_0 = (S^3\setminus\nu(J))\times I$.
The Seifert fibered structure on the torus knot exterior $E$ extends across the disk exterior, and the quotient orbifold is the solid pillowcase
$$B^3(p, q) = D^2(p, q)\times I,$$
whose total space is a 3--ball and whose cone points form two neatly embedded arcs.
See Figure~\ref{fig:pillow}, in which $(\partial E\times I)^*$ is $\mu^*$, and the right half of the boundary is $(E\times\{1\})^*$, and the left half is $(E\times\{0\})^*$ with reversed orientation.
This perspective makes clear the fact that $\partial Z_0$ is the union of the exterior $\overline E$ of the left-handed torus knot with the exterior $E$ of the right-handed torus knot along their identified boundary.
By construction, we get a Seifert fibered structure on $Y_Q = \partial Z_0 = \overline E\cup E$ with base orbifold $S^2(p, q, p, q)$, where $Y_Q$ denotes the result of 0--framed Dehn surgery on $Q$.
We denote the 2--sphere boundary of $B^3(p,q)$ by $S$.

The inclusion $E\hookrightarrow Z_0$ induces an isomorphism of fundamental groups, so the presentations calculated above can be regarded as presentations for $\pi_1(Z_0,c_0^*)$.
Thus, we have
\begin{align}
\label{eqn:pi1z0}
\pi_1(Z_0,c_0^*) = \left \langle u, v, h \mid u^p=h^{-b}, v^q=h^{-a}, uh=hu, vh=hv \right\rangle= \left\langle x, y \mid x^p = y^q \right\rangle.
\end{align}
The meridian $\mu_J$ includes into $\partial Z_0 = Y_Q$ as the Dehn surgery dual and  $[\mu_J] = vu = y^ax^b$.

We emphasize that $\eta_1$ and $\eta_2$ are curves in $P\subset E\subset \partial Z_0$ that are sections of the curves $\eta_1^*$ and $\eta_2^*$ whose loop homotopy classes $u$ and $v$ are indicated in Figure~\ref{fig:pillow}; cf Remark~\ref{rmk:identify_1}.

\subsection{The homotopy class of a slope on the pillowcase}
\label{subsec:bridge_relator}
\ 

Let $S'$ denote the 4--punctured sphere resulting from the removal of open neighborhoods of the cone points of the pillowcase.
Recalling the pair of pants $P^* \subset D^2(p, q)$, we have $S' = \overline {P^*}\cup_{\mu^*}P^*$.
Let $B'$ denote the genus--2 handlebody resulting from the removal of the arcs of cone points from the solid pillowcase, so $B' = P^*\times I$.
By the previous discussion, we have
$$\pi_1(B',c_0^*) = \langle u, v \mid - \rangle,$$
as depicted in Figure~\ref{fig:pillow}.
Here we are making the change in perspective previewed by Remark~\ref{rmk:identify_1}: we regard $u$ and $v$ as $[\eta_1^*]$ and $[\eta_2^*]$, respectively, and we regard $\eta_i^*\subset P^*\subset S'\subset B'$.

Our next task is to describe simple closed curves on the punctured pillowcase $S'$ as elements of $\pi_1(B',c_0^*)$.
Such simple closed curves are called \emph{slope-curves}, because there is one such curve for each $c/d\in\Q\cup\{\infty\}$.  We will be interested in slope-curves corresponding to $c/d$ when $c$ is even: for example, the slope-curve $\lambda_{10/7}$ is the boundary of a regular neighborhood of the arc with slope 10/7 shown in Figure~\ref{fig:slope}.  Henceforth,  we will abbreviate `slope-curve' to simply `slope',  as is common in the literature.

\begin{figure}[h!]
	\begin{subfigure}{.57\textwidth}
		\centering
		\includegraphics[width=.8\linewidth]{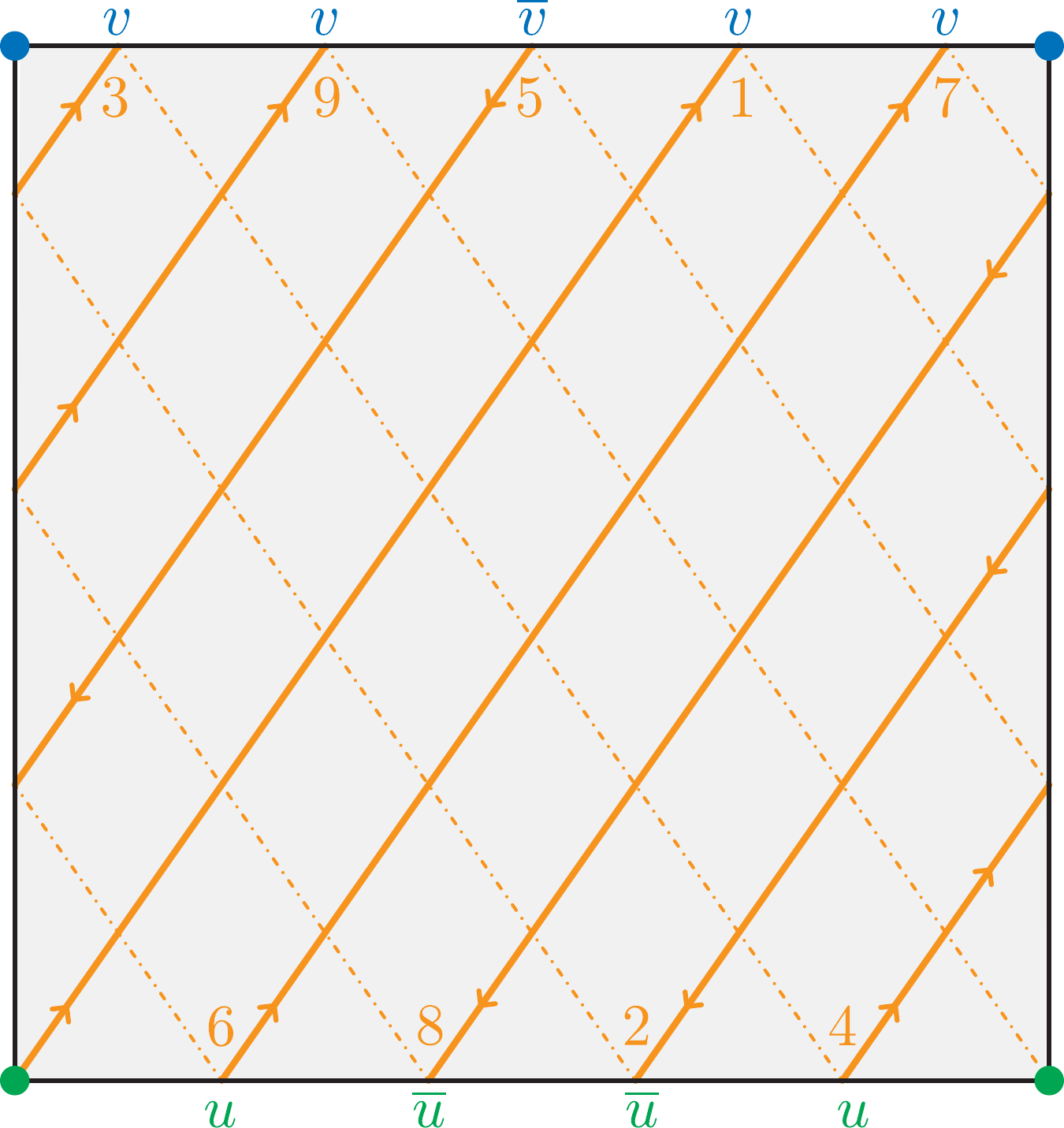}
		\caption{The arc $a_{10/7}$ on the boundary of the solid pillowcase}
		\label{fig:slope}
	\end{subfigure}%
	\begin{subfigure}{.43\textwidth}
		\centering
		\includegraphics[width=.8\linewidth]{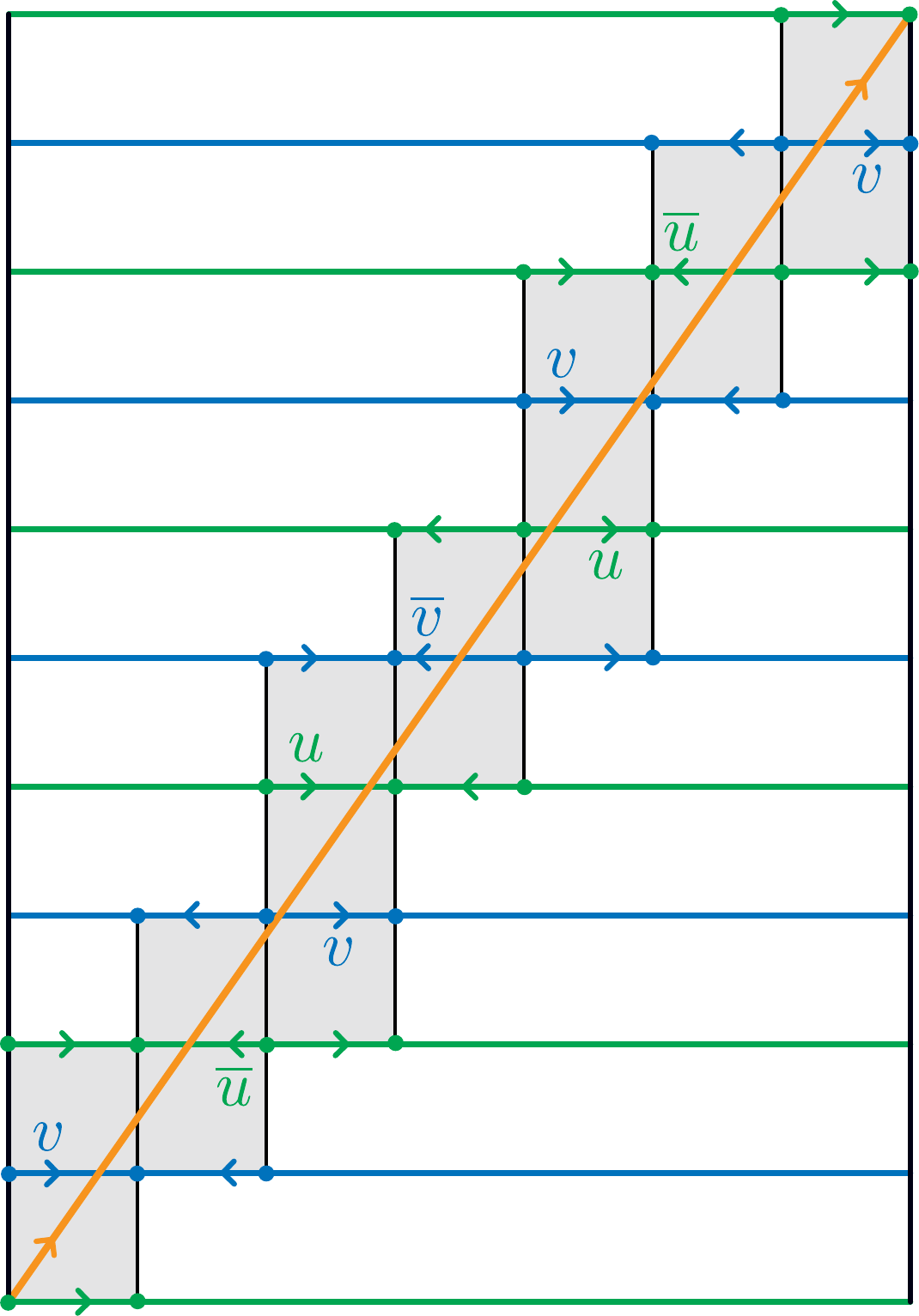}
		\caption{The arc $\widetilde a_{10/7}$ on the grid $R$}
		\label{fig:stairs}
	\end{subfigure}%
	\caption{Determining the homotopy class of a slope-curve}
	\label{fig:slope_stairs}
\end{figure}

Let $\lambda$ be a slope on $S'$.
We will describe $[\lambda]\in\pi_1(B',c_0^*)$ as a word in $u$ and $v$.
This characterization was given (without proof) by Hartley~\cite{Har_79}, who used it to study the Alexander polynomials of 2--bridge knots.

\begin{proposition}
\label{prop:slope_relation}
	Let $c/d\in\Q$ with $c$ even.
	Let $\lambda_{c/d}$ denote the corresponding slope on the punctured pillowcase $S'$.
	For $i\in\N$,  let $\varepsilon_i = (-1)^{\lfloor\frac{id}{c}\rfloor}$.
	Then,
	$$[\lambda_{c/d}] = u\cdot\omega_{c/d}(u,v)\cdot u^{-1}\cdot \omega_{c/d}(u,v)^{-1}$$
	as an element of $\pi_1(B', c_0^*)$,  where $\omega_{c/d}(u,v) = v^{\varepsilon_1} u^{\varepsilon_2}\cdots v^{\varepsilon_{c-1}}$.
\end{proposition}

\begin{proof}
	By definition, the slope $\lambda_{c/d}$ is the boundary of a neighborhood of an arc $a_{c/d}$ running from $\bar c_1^*$ to $c_1^*$ on $S$.
	From this, we see that $[\lambda_{c/d}] = u\cdot\omega_{c/d}(u,v)\cdot  u^{-1}\cdot \omega_{c/d}(u,v)^{-1}$ for some word $\omega_{c/d}(u,v)$ in the generators $u$ and $v$.
	The word $\omega_{c/d}$ will alternate between instances of $u^\pm$ and $v^\pm$, picking up one instance every time the arc $a_{c/d}$ passes across either the top seam ($u^\pm$) or the bottom seam ($v^\pm$) of the pillowcase.
	
	The combinatorics of these passages can be understood by cutting the pillow open along three seams, and tiling a $c/2\times d$ rectangular grid $R$ with the resulting objects; see Figure~\ref{fig:stairs}.
	(Here we assume for simplicity that $c>0$; see Remark~\ref{rmk:negative_slope}.)
	The horizontal grid-lines of $R$ will be alternately labeled with $u^\pm$ or $v^\pm$, with the exponent alternating as one moves left and right and the variable alternating as one moves up and down.
	The top and bottom edges of the grid are labeled $u$.
	
	The arc $a_{c/d}$ will lift to an arc $\widetilde a_{c/d}$ from the bottom-left corner of $R$ to the top-right corner.
	To recover $\omega_{c/d}(u,v)$, record the sequential intersections of $\widetilde a_{c/d}$ with the horizontal grid-arcs labeled by $u^\pm$ and $v^\pm$.
	The variables will alternate, and the signs of the exponents will be given by the sequence $(\varepsilon_i)_{i=1}^{c-1}$.
	
	For example, by regarding either Figure~\ref{fig:slope} or~\ref{fig:stairs}, one finds that $\omega_{10/7} = v u^{-1} v u v^{-1} u v^{-1} u v$.
	The grey-shaded regions are the lifts of the pillowcase that intersect $\widetilde a_{10/7}$.
\end{proof}

\begin{remark}
\label{rmk:negative_slope}
	If $c/d$ is negative, then the analysis described above can be carried out in the same way, where the arc $a_{c/d}$ is the mirror of $a_{-c/d}$ across a vertical line in the page.
	This is a symmetry of the pillowcase, so none of the combinatorics are altered.
	Practically, $\widetilde a_{-c/d}$ will run from the bottom-right corner to the top-left corner in $R$, but the sequential encounters will be the same.
	It follows that $\omega_{-c/d} = \omega_{c/d}$.
	Alternatively, if $c/d<0$, then the algorithm works by setting $\varepsilon_i = \lceil\frac{id}{c}\rceil$.
\end{remark}

\begin{remark}
\label{rmk:large_numerator}
	The left-handed Dehn twist $\tau_0$ along the slope $\lambda_0$ of the pillowcase (shown as the black equator in Figure~\ref{fig:pillow}) extends over the solid pillowcase, so two slopes that are related by a power of this Dehn twist will represent the same word in $\pi_1(B',c_0^*)$.
	Since $\tau_0^n\left(a_{c/d}\right) = a_{c/(d-2nc)},$
	we have for any $n \in \Z$ that 
	$$\omega_{c/(d-2nc)}(u,v) = \omega_{c/d}(u,v).$$
\end{remark}

\begin{example}
\label{ex:twist_words}
	For a simple example, consider the words corresponding to the slopes corresponding to $\frac{c}{d} = \frac{2n}{1}$ for $n\in\Z$.
	In this case, 
	$$(\varepsilon_i)_{i=1}^{2n-1} = \left((-1)^{\lfloor\frac{i}{2n}\rfloor}\right)_{i=1}^{2n-1} =  (1, 1, \ldots, 1).$$
	So we have $\omega_{2n/1}(u,v) = (vu)^n$, and
	$$[\lambda_{2n/1}] = u(vu)^nu^{-1}(u^{-1}v^{-1})^n = (uv)^n(u^{-1}v^{-1})^n.$$
	Note that this is independent of the sign of $n$.
\end{example}

\subsection{Identifying the kernels}
\label{subsec:kernels}
\ 

Above, we described a Seifert fibered structure on $Y_Q$ whose base orbifold is the pillowcase $S = S^2(p, q, p, q)$.
The manifold $Y_Q$ also has the structure of a closed surface-bundle over $S^1$, where 
the fiber is the closed surface $\widehat F$ of genus $g = (p-1)(q-1)$ obtained by capping off a Seifert surface for $Q=Q_{p,q}$ in the 0--surgery, and the monodromy $\widehat\varphi\colon\widehat F\to\widehat F$ is periodic of order $pq$.
The surface-fiber $\widehat F$ is  equivariant with respect to the circle-action, and its quotient is the pillowcase $S$.  The induced quotient map $\pi\colon\widehat F\to S$ is a covering space map of degree $pq$ with four branched points (the cone points of the pillowcase), whose local branching degrees are either $p$ or $q$.
We refer the reader to~\cite[Section~4]{MeiZup_Qpq} for a thorough discussion of how the surface-fibers interact with the circle-actions in the present set-up.
In particular, by Proposition~4.10 of~\cite{MeiZup_Qpq}, when $c$ is even, the preimage $\pi^{-1}(\lambda_{c/d})$ is a collection $\Ll_{c/d}$ of $pq$ curves on $\widehat F$ that cut the surface-fiber into planar pieces.
The curves of $\Ll_{c/d}$ are pairwise isotopic in $Y_Q$ (see the proof of~\cite[Lemma~5.1]{MeiZup_Qpq}); in fact, there is a torus in $Y_Q$ that is the orbit of a section of $\lambda_{c/d}$ and that intersects $\widehat F$ in $\Ll_{c/d}$.

By~\cite[Section~8]{MeiZup_Qpq},  the exterior $Z_{c/d} = B^4\setminus\nu(D_{c/d})$ is built by attaching $pq$ 2--handles (framed by $\widehat F$) to $Y_Q \times I$ along $\Ll_{c/d}$ before capping off with $pq+1$ 3--handles and a 4--handle. 
Since the curves of $\Ll_{c/d}$ are isotopic in $Y_Q$,  the disk exterior $Z_{c/d}$ is diffeomorphic to the manifold built from $Y_Q \times I$ by attaching a single 2--handle along any component $V_{c/d}$ of $\Ll_{c/d}$ before capping off with two 3--handles and a 4--handle.

These remarks establish the following.

\begin{proposition}
\label{prop:generator}
	The inclusion-induced map $i_{c/d}\colon \pi_1(Y_Q)\to \pi_1(Z_{c/d})$ is an epimorphism, and $\ker(i_{c/d})$ is normally generated by $[V_{c/d}]$ in $\pi_1(Y_Q)$.
\end{proposition}

In Subsection~\ref{subsec:disk_group}, we identified the presentation
$$\pi_1(Z_0,c_0^*) = \left \langle u, v, h \mid u^p=h^{-b}, v^q=h^{-a}, uh=hu, vh=hv \right\rangle,$$
where $u$ and $v$ were represented by $\eta_1$ and $\eta_2$, respectively, via the inclusion $P\subset E\subset Y_Q\subset Z_0$.
In Subsection~\ref{subsec:bridge_relator}, we characterized the homotopy classes of slope-curves on $S'\subset S^2(p,q,p,q)$ as words in $u$ and $v$.
As discussed in Remark~\ref{rmk:identify_1}, we have identified $P^*$ with its section $P$, which allows us to identify $S'= \overline{P^*} \cup P^*$ with $\overline{P} \cup P \subset \overline{E} \cup E = Y_Q \subset Z_0$.
This identification justifies the simultaneous use of $u$ and $v$ in these \emph{a priori} distinct settings.
Moreover, we see that $\lambda_{c/d}$ is identified with $V_{c/d}$ via this set-up.
So, by describing $[\lambda_{c/d}]$ as a word in $u$ and $v$ in $\pi_1(S',c_0^*)$, we have also described $[V_{c/d}]$ as a word in $u$ and $v$ in $\pi_1(Z_0,c_0^*)$.
(The use of the same notation for the base-point is reflective of the overarching philosophy here.)
We summarize this with the following diagram:
$$
\begin{tikzcd}
& V_{c/d} \arrow[r,hook] &  Y_Q \arrow[r,hook] \arrow[d,two heads] & Z_0 \arrow[d,two heads] \\
 \lambda_{c/d} \arrow[r,hook] \arrow[ur,leftrightarrow] & S' \arrow[r,hook] \arrow[ur,hook] & S^2(p,q,p,q)\arrow[r,hook] & B^3(p,q)
\end{tikzcd}
$$
In light of this, we can characterize when the image of $[V_{c/d}]$ is  trivial in $\pi_1(Z_0,c_0^*)$.

\begin{proposition}
\label{prop:nontrivial_class}
	Let $c/d\in\Q$ with $c$ even.
	The class $[V_{c/d}]$ is trivial in $\pi_1(Z_0)$ if and only if $c/d=0$.
\end{proposition}

\begin{proof}
	By the above discussion, we have that in $\pi_1(Z_0,c_0^*)$
	$$[V_{c/d}] = u\cdot\omega_{c/d}(u,v)\cdot u^{-1}\cdot[\omega_{c/d}(u,v)]^{-1},$$
since $[\lambda_{c/d}]$ has this form in $\pi_1(B',c_0^*)$, by Proposition~\ref{prop:slope_relation}.
	Making the change of variable $u\mapsto x^b$ and $v\mapsto y^a$, we have
	$$[V_{c/d}] = x^b\cdot\omega_{c/d}(x^b,y^a)\cdot x^{-b}\cdot[\omega_{c/d}(x^b,y^a)]^{-1}$$
	in $\pi_1(Z_0,c_0^*)$.
	
	Consider the epimorphism $\rho\colon \pi_1(Z_0)\twoheadrightarrow \Z_p\ast\Z_q$ defined by adding the relations $x^p = y^q = 1$.
	Since $\gcd(p,b)=\gcd(q,a)=1$, and since the terms of $\omega(x^b,y^a)$ alternate between instances of $x^{\pm b}$ and $y^{\pm a}$, it follows that $\rho([V_{c/d}])$ has syllable length $2c$ in the free product and cannot be cyclically reduced.
	If $[V_{c/d}]$ were trivial in $\pi_1(Z_0,\ast)$, then its image $\rho([V_{c/d}])$ would be trivial in $\Z_p\ast\Z_q$, a contradiction, unless $c=0$.
	
	Conversely, if $c=0$, then $\lambda_0$ bounds a disk in $W$, so $[\lambda_0]$ is trivial in $\pi_1(B',c_0^*)$, hence $[V_0]$ is trivial in $\pi_1(Z_0,c_0^*)$.
\end{proof}

We are finally ready to assemble the proof of Theorem~\ref{thm:kernels}.

\begin{proof}[Proof of Theorem~\ref{thm:kernels}]
	Consider the following inclusion-induced epimorphisms:
	$$\pi(Q) \stackrel{j}{\twoheadrightarrow} \pi_1(Y_Q) \stackrel{i_{c/d}}{\twoheadrightarrow}\pi(D_{c/d}).$$
	The epimorphism is induced by including $S^3\setminus\nu(Q)$ into $Y_Q$, its 0--framed Dehn filling.
	This adds a relation to $\pi(Q)$.
	The map $i_{c/d}$ is as described in Proposition~\ref{prop:generator}.
	Since these maps are induced by inclusion, $\iota_{c/d} = i_{c/d}\circ j$.
	We observe that $\ker(\iota_{c/d})= \ker(\iota_{c'/d'})$ if and only if $\ker(i_{c/d})= \ker(i_{c'/d'})$, since these maps are surjective.
	
	Suppose that $\ker(\iota_{c/d}) = \ker(\iota_{c'/d'})$.
	By the above observation, this implies $\ker(i_{c/d}) = \ker(i_{c'/d'})$.
	By Proposition~\ref{prop:generator}, we know that $\ker(i_{c'/d'})$ is normally generated by $[V_{c'/d'}]\in\pi_1(Y_Q)$, so $\ker(i_{c/d})$ is normally generated by $[V_{c'/d'}]\in\pi_1(Y_Q)$.
	By~\cite[Proposition~8.5]	{MeiZup_Qpq},  there exists a diffeomorphism $\Psi_{c/d}\colon Z_{c/d}\to Z_0$ such that $\Psi_{c/d}(V_{c/d}) = V_0$ and $\Psi_{c/d}(V_{c'/d'})=V_{c''/d''}$ for some $c''/d''$ that equals $0$ if and only if $c/d$ equals $c'/d'$.
	It follows that $\ker(i_0)= (\Psi_{c/d})_*(\ker(i_{c/d}))$ is normally generated by $[V_{c''/d''}]=[\Psi_{c/d}(V_{c'/d'})]\in\pi_1(Y_Q)$.
	In particular, we have that $[V_{c''/d''}]\in\ker(i_0)$.
	This means that $[V_{c''/d''}]$ is trivial in $\pi_1(Z_0)$.
	By Proposition~\ref{prop:nontrivial_class}, this implies $c''=0$, which means that $c'/d'=c/d$, as desired.
\end{proof}

\section{Other examples}
\label{sec:examples}

In this section, we describe general constructions of knots with many disks that are LK-nonisotopic, as alluded to in the introduction.

\subsection{Invertible concordances}
\label{subsec:invertible}
\ 

In this section, we give examples of hyperbolic knots with infinitely many ribbon disks that are pairwise nonisotopic modulo local knotting.
Any concordance from $K$ to $J$ can be used to turn a slice disk for $J$ into a slice disk for $K$, but it is not guaranteed in general that distinct slice disks for $J$ will give rise to distinct slice disks for $K$.
For example, if $K$ is the unknot, then any two slice disks for $J$ will become  LK-isotopic rel.~boundary as slice disks for $K$.
However,  we can control the situation when the concordance from $K$ to $J$ is \emph{invertible} in the sense of Sumners~\cite{Sumners}; see the proof of Proposition~\ref{prop:invert} for the definition.

\begin{proposition}
\label{prop:invert}
	Let $C$ be an invertible concordance from $K$ to $J$. 
	Let $D_1$ and $D_2$ be slice disks for $J$. 
	Then $C \cup_J D_1$ and $C \cup_J D_2$ are isotopic rel. boundary modulo local knotting if and only if $D_1$ and $D_2$ are.  
\end{proposition}

\begin{proof}
	One direction is immediate.
	For the converse assume that $C \cup_J D_1$ and $C \cup_J D_2$ are  LK-isotopic rel. boundary.  
	It follows that there exist 2--knots $S_1$ and $S_2$ such that the $(C \cup_J D_i) \# S_i$ are isotopic rel. boundary.
	We can assume that the connected sum region of each of these occurs in the interior of $D_i$. 

	Let $C'$ be the inverse of $C$, so $C'$ is  a concordance from $K$ to $J$ such that $\overline C' \cup_K C$ is isotopic rel. boundary to the identity concordance from $J$ to $J$.
	It follows that 
	\[\overline C' \cup_K \left( (C \cup_J D_1) \# S_1 \right) \stackrel{\partial}{\simeq} 
	\left (\overline C' \cup_K C\right) \cup_J \left( D_1  \# S_1\right) \stackrel{\partial}{\simeq} 
	D_1 \# S_1
	\]
	is isotopic rel. boundary to
	\[\overline C' \cup_K \left( (C \cup_J D_2) \# S_2 \right) \stackrel{\partial}{\simeq} 
	\left (\overline C' \cup_K C \right) \cup_J \left( D_2  \# S_2 \right) \stackrel{\partial}{\simeq} 
	D_2 \# S_2.
	\]
So $D_1$ and $D_2$ are LK-isotopic rel. boundary.
\end{proof}

\begin{corollary}
\label{coro:hyp_fib}
	There exist infinitely many fibered, hyperbolic knots each of which bounds infinitely many fibered, ribbon disks that are pairwise nonisotopic modulo local knotting. 
\end{corollary}

\begin{proof}
	We first give a single example of a knot satisfying the conclusions of the corollary, before indicating at the end of the proof how to generalize the construction to yield an infinite family.

\begin{figure}[h!]
	\centering
	\includegraphics[width=.7\linewidth]{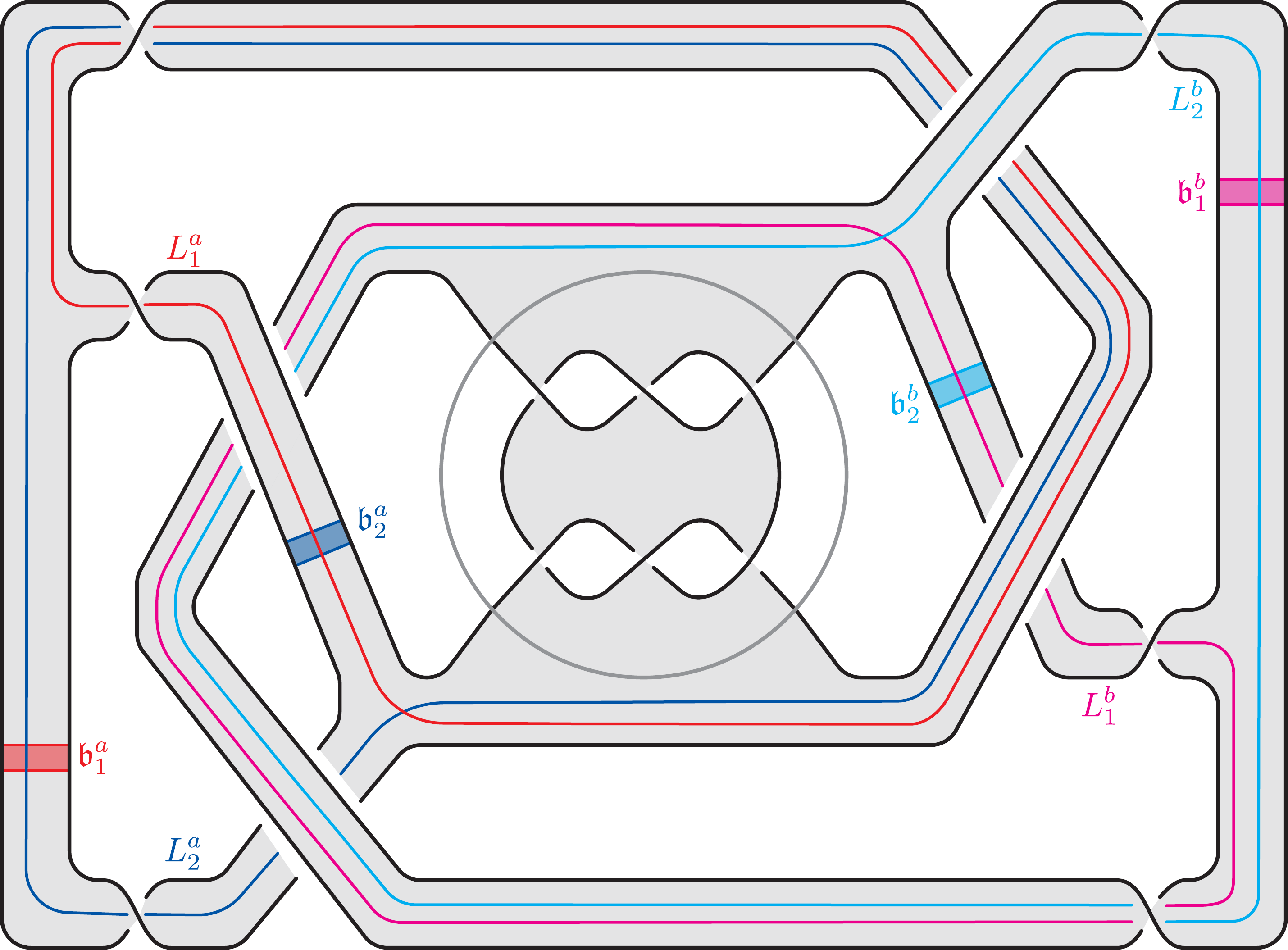}
	\caption{A fibered, hyperbolic knot bounding infinitely many fibered, ribbon disks that are pairwise nonisotopic modulo local knotting. }
	\label{fig:hyp_fib}
\end{figure}

	Let $K$ be the knot shown in Figure~\ref{fig:hyp_fib} together with a grey genus-four Seifert surface $F$. 
	Our first task is to establish that $K$ is fibered with fiber-surface $F$: we will do so by showing that $F$ can be obtained by plumbing of two fiber surfaces. 
Consider the region of $K$ inside the grey circle; this 2--strand tangle $T$ is also shown in Figure~\ref{fig:plumb3}.
	If we modify $K$  and $F$ by replacing $T$ with the 2--strand trivial tangle $T'$ and corresponding surface shown in Figure~\ref{fig:plumb1}, 
 the result is a square knot $K$' with a genus-two Seifert surface $F'$,  which since it is genus-two must be the fiber surface for $K'$. 
	Viewed in the other direction, the pair $(K,F)$ is obtained from $(K',F')$ by a local modification changing Figure~\ref{fig:plumb1} to Figure~\ref{fig:plumb3}. 
	This modification can be  realized by taking the square knot $K''$ shown with its fiber-surface $F''$ in Figure~\ref{fig:plumb2} and plumbing this surface onto the surface bounded by $T'$ in Figure~\ref{fig:plumb1}.
	The plumbing is accomplished by gluing the two surfaces along the indicated green quadrilaterals; see Stallings~\cite{stallings} for details.
	Since $F'$ is a fiber-surface for the square knot $K'$ and $F''$ is a fiber-surface for the square knot $K''$, it follows that $F$ is a fiber-surface for $K$~\cite[Theorem~1]{stallings}.
	This establishes that $K$ is fibered with fiber-surface $F$.

\begin{figure}[h!]
	\begin{subfigure}{.25\textwidth}
		\centering
		\includegraphics[width=1.1\linewidth]{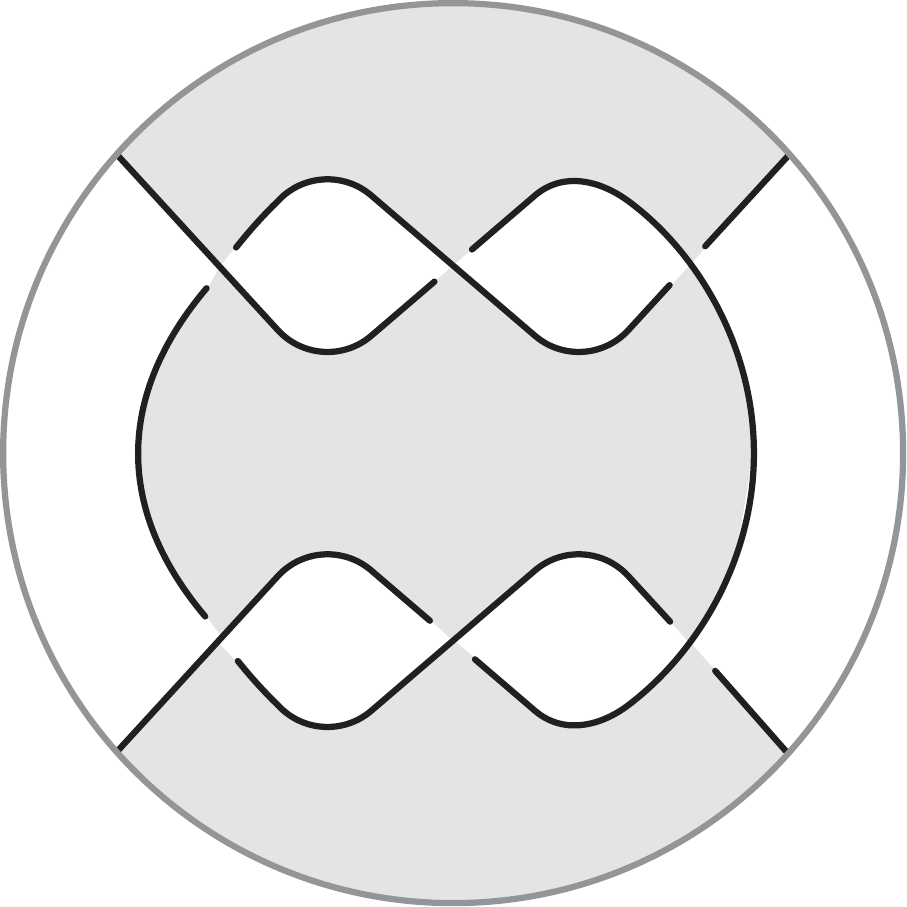}
		\caption{The tangle $T$ and a portion of the surface $F$}
		\label{fig:plumb3}
	\end{subfigure}%
	\hspace{1cm}
	\begin{subfigure}{.25\textwidth}
		\centering
		\includegraphics[width=1.1\linewidth]{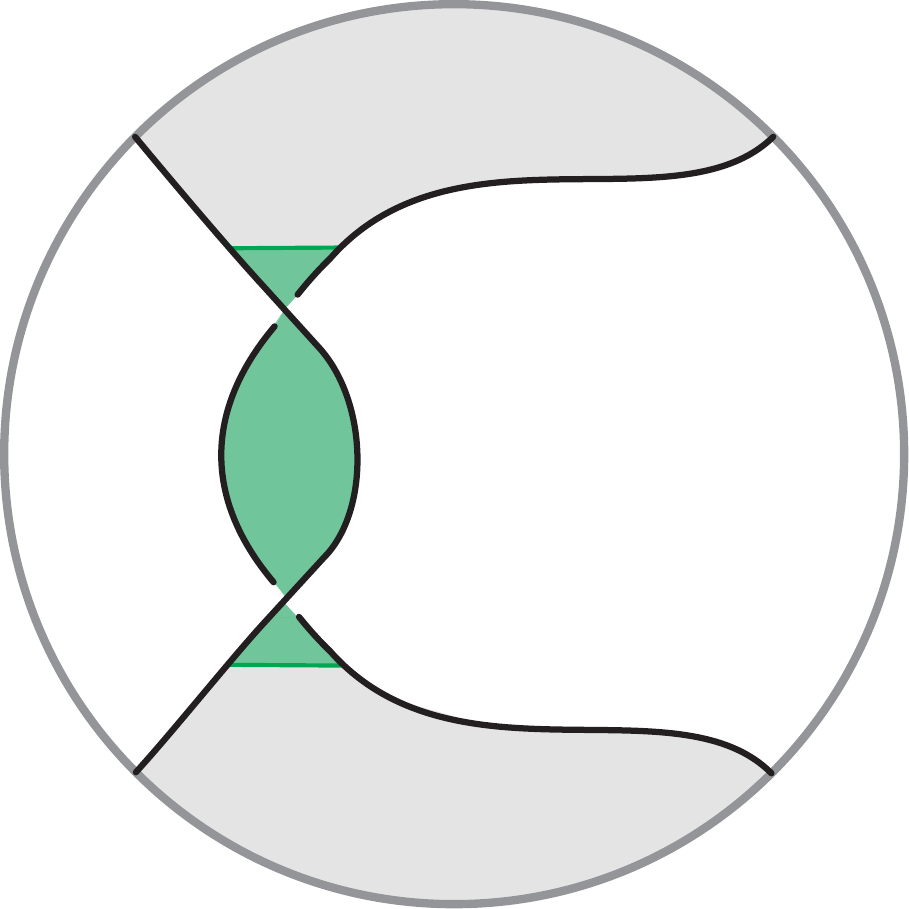}
		\caption{The tangle $T'$ and a portion of the surface $F'$}
		\label{fig:plumb1}
	\end{subfigure}%
	\hspace{1cm}
	\begin{subfigure}{.25\textwidth}
		\centering
		\includegraphics[width=1.1\linewidth]{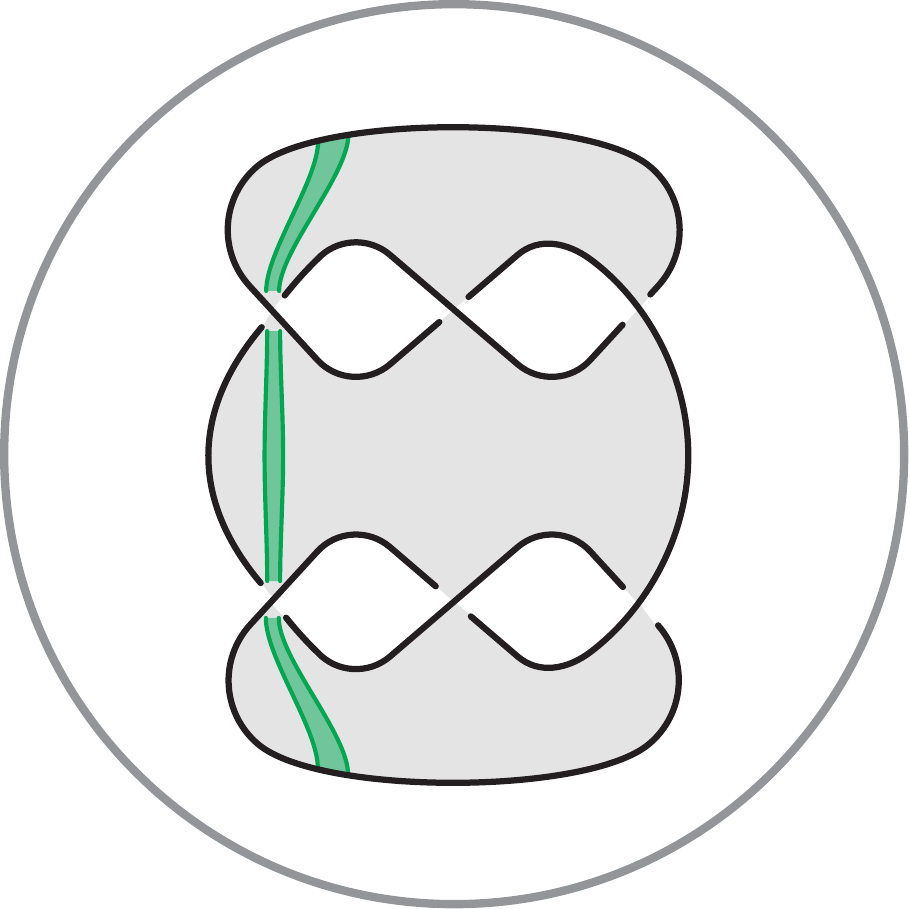}
		\caption{The surface $F''$ \\ \ }
		\label{fig:plumb2}
	\end{subfigure}%
	\caption{The surface $F$ is obtained by plumbing the surface $F''$ onto the surface $F'$ along the green quadrilateral, which changes the tangle $T'$ to the tangle $T$.}
	\label{fig:plumbing}
\end{figure}
	
	Using SnapPy~\cite{SnapPy} in SageMath~\cite{sage},  we verify that $K$ is hyperbolic using the command	\verb|N.verify_hyperbolicity()|, where $N = S^3\setminus K$.

	Let $L_1=L_1^a \cup L_1^b$ be the link on $F$ consisting of the red and magenta curves  and let $L_2=L_2^a \cup L_2^b$ be the link on $F$ consisting of the blue and cyan curves.
	Note that $L_i$ is a 2--component unlink that is 0--framed by $F$,  and so there is a concordance $C_i$  defined by surgering $F$ along the two curves of $L_i$.
	One can check that $K$ can be slid over $L_i$ to give $Q = Q_{3,2}$,  and so $C_i$ is a concordance from $K$ to $Q$.
	The result of surgering $F$ along $L_i$ is a genus-two Seifert surface $F'_i$  for $Q$, hence a fiber-surface for $Q$.
	Let $M_i$ denote the trace of this surgery, so $M_i$ is a compression body with
	$$\partial M_i = F \cup C_i \cup \overline F_i'.$$
	The concordance $C_i$ is also described by the bands $\frak{b}_i=\frak{b}_i^a \cup \frak{b}_i^b$, together with capping off $L_i$ with the standard disks for the unlink, since resolving $\frak b_i$ cuts $F$ into a genus-two surface with boundary the split union $Q\sqcup L_i$.
	Since the bands $\frak b_i$ are contained in $F$, they can be perturbed to be transverse to the fibration of $K$.
	It follows from~\cite[Theorem~1.9]{Mil_21} that the ribbon concordance $C_i$ is fibered by the compression body $M_i$.
	Hence, $C_1 \cup D_{c/d}$ is a fibered, homotopy ribbon disk for $Q$ whenever $c/d \in \Q$ has $c$ even.
	
	We remark in passing that it is possible to directly verify that $C_i$ is invertible with $C_1' = C_2$ using the bands $\frak b_1$ and $\frak b_2$ and the band calculus of~\cite{Swenton}, following the technique of~\cite[Section~4]{LivMei_15}, but we will give another approach here.
	Let $M = M_1\cup_F \overline M_2$, so $M$ is a 3--manifold with
	$$\partial M = F_2'\cup C_2\cup\overline C_1\cup \overline F_1',$$
	and $F\subset M$ is a Heegaard surface, splitting $M$ into the compression bodies $M_1$ and $M_2$, with Heegaard diagram $(F; L_1, L_2)$.
	We see that $L_1$ and $L_2$ are dual on $F$, so this Heegaard splitting of $M$ is twice-stabilized.
	It follows that $M\cong F' \times I$  and  $C_1\cup \overline C_2 \cong Q\times I$.

	Proposition~\ref{prop:invert} now  combines with Theorem~\ref{thm:main_IRB} to imply that $C_1 \cup D_{c/d}$ contains an infinite collection of fibered,  homotopy-ribbon disks for $K$ that are pairwise LK-nonisotopic rel. boundary.
	Since $C_1$ is ribbon and each $D_{2/(2m+1)}$ is ribbon~\cite[Theorem 1.6]{MeiZup_disks},  so is each $C_1 \cup D_{2/(2m+1)}$.
	Since $K$ is hyperbolic,  $\textsc{Sym}(K)$ is finite~\cite[Theorem~10.5.3]{Kaw_96},  and so there is an infinite subcollection of these fibered, ribbon disks for $K$ that are pairwise LK-nonisotopic.

To generalize $K_1:=K$ to an infinite family $\{K_m\}_{m\in\N}$ of knots, use $K_m'' := Q_{2m+1,2}$ (along with its fiber-surface $F''$) instead of $K_1'':=K'' = Q_{3,2}$, which is shown in Figure~\ref{fig:plumb2}, in the plumbing operation described at the beginning of the proof.
The result is that the two 3--twist regions of $T$ become two $(2m+1)$--twist regions. 
The arguments given above show that $K_m$ has a Seifert surface obtained by plumbing the fiber-surface for $Q_{2m+1,2}$ to that of $Q_{3,2}$, and hence is fibered,  and also that it has an invertible, fibered, ribbon concordance to $Q_{2m+1,2}$.
This will imply as above that $K_m$ has infinitely many fibered, ribbon disks that are pairwise LK-nonisotopic,  so long as $K_m$ is hyperbolic.
For this, consider the link $L = K_1\cup\alpha_1\cup\alpha_2$, where $\alpha_1$ and $\alpha_2$ are unknots linking the upper and lower 3--twist regions of $T\subset K_1$, respectively.
The knot $K_m$ is obtained  from $K_1$ by performing $1/(m-1)$--framed Dehn surgery on $\alpha_1$ and $-1/(m-1)$--framed Dehn surgery on $\alpha_2$.
Using SnapPy in SageMath as before, we can verify that $L$ is hyperbolic, so at most finitely many values of $m$ will yield nonhyperbolic fillings~\cite{Thu_82}.
In other words,  infinitely many of the $K_m$ will be hyperbolic, as desired.
\end{proof}

\begin{corollary}
\label{coro:pseudoA}
	There exist infinitely many distinct pseudo-Anosov mapping classes each of which has infinitely many distinct handlebody extensions. 
\end{corollary}

\begin{proof}
Let $\widehat\varphi_m$ denote the closed monodromy of the knot $K_m$ from the proof of Corollary~\ref{coro:hyp_fib}.
As was done in that proof, we can use SnapPy in SageMath to verify that $S^3_0(K_m)$ is hyperbolic and hence $\widehat\varphi_m$ is pseudo-Anosov for infinitely many values of $m$.
For the rest of the proof, fix some such value of $m$.

Let $\Ll_{c/d}' = L_1\cup \Ll_{c/d}$, where the $\Ll_{c/d}\subset F_1'$ is the link introduced in Subection~\ref{subsec:kernels}, but now we view $\Ll_{c/d}$ as lying on the closed fiber $\widehat F$ for $K_m$ in the obvious way.
	The link $\Ll_{c/d}'$ defines a handlebody extension of the closed monodromy of $K_m$, with the handlebody fiber given by $H'_{c/d}:=M_1\cup H_{c/d}$, where $H_{c/d}$ is the handlebody fiber of the disk $D_{c/d}$ for $Q_{2m+1,2}$.
	As in the proof of Proposition~\ref{prop:invert}, if $H_{c/d}'$ and $H_{c'/d'}'$ are diffeomorphic rel. boundary, then so are $M_2\cup H_{c/d}'$ and $M_2\cup H_{c'/d'}'$, but these are simply $H_{c/d}$ and $H_{c'/d'}$, by the earlier observed invertibility.
	By~\cite[Theorem~1.7]{MeiZup_Qpq}, these latter handlebodies are diffeomorphic rel. boundary if and only if $c'/d' = c/d$.
\end{proof}

\subsection{Satellite operations}
\label{subsec:satellite}
\ 

In this section, we obtain more  examples of knots bounding infinitely many disks that are pairwise LK-nonisotopic rel. boundary by using the satellite construction.
In particular,  the `satellite the concordance' argument commonly used to show that $K \mapsto P(K)$ induces a well-defined map on knots modulo concordance also gives a recipe for combining slice disks for each of $K$ and $P(U)$ into a slice disk for $P(K)$  as follows. 

 \begin{definition}~\label{defn:satdisk}
 Let $P= P(U) \cup \eta$ be a \emph{pattern} -- i.e. an ordered 2--component link where $\eta$ is unknotted and $P(U)$ is described by a map $i_P \colon S^1 \to S^1 \times D^2 \cong S^3 \setminus \nu(\eta)$,  where we use the Seifert framing on $\eta$ in the latter identification.  
Let $K$ be a knot with slice disk $D$,  let $C_D$ be the concordance from $K$ to $U$ obtained by puncturing $D$,  and let 
\[ i_{\overline{\nu}(C_D)} \colon I \times S^1 \times D^2 \cong C_D \times D^2 \to S^3 \times I\]
denote an embedding of the tubular neighborhood of $C_D$,  where we require $i_{\overline{\nu}(C_D)}$ to restrict to the Seifert framings on the tubular neighborhoods of $P(K)$ and $P(U)$.
The \emph{satellite concordance $P(C_D)$} from $P(K)$ to $P(U)$ is the image of the following composition:
\[ I \times S^1 \xrightarrow{ \id_I \times i_P} I \times S^1 \times D^2 \xrightarrow{ i_{\overline{\nu}(C_D)}} S^3 \times I.\]

If $P(U)$ has a slice disk $\Delta$,    the  \emph{satellite slice disk} $P_{\Delta}(D)$ for $P(K)$ is
\[ P_{\Delta}(D):= P(C_{D}) \cup_{P(U)} \Delta \subset (S^3 \times I) \cup B^4 \cong B^4. \]
\end{definition}

We note for later that the exterior $E_{P_{\Delta}(D)}$ is obtained by gluing the exteriors $E_{D}$ and $E_{\Delta}$ together by identifying the surgery solid torus in $\partial E_D=S^3_0(K)$ with a neighborhood of $\eta$ in $ \partial E_{\Delta}=S^3_0(P(U))$. 
We also remark that Definition~\ref{defn:satdisk} generalizes the definition of `satellite disk' given in~\cite{GuthEtc}, where the additional assumption that $P(U)=U$ is included.

We now give some conditions that ensure that distinct disks for $P(U)$ or for $K$ lead to distinct disks for $P(K$). 
The following result comes from \cite{GuthEtc}  and essentially results from considering the Seifert-Van Kampen theorem applied to  $E_{P(K)}=E_K \cup_{T^2} E_P$ and $E_{P_{\Delta}(D)}= E_D \cup_{S^1 \times D^2}  E_\Delta$.

\begin{proposition}\label{prop:KtoP(K)}
Let $P=P(U) \cup \eta$ be a pattern.  Suppose that $\Delta$ is a slice disk for $P(U)$ such that the image of the class of $\eta$ in $\pi(\Delta)$ is of infinite order. 
Let $D_1$ and $D_2$ be  slice disks for $K$ such that 
 \[\ker \left( \pi(K) \to \pi(D_1)\right) \neq \ker \left( \pi(K) \to \pi(D_2)\right).\]
Then  $P_{\Delta}(D_1)$ and $P_{\Delta}(D_2)$ are slice disks for $P(K)$  such that 
\[\ker \left( \pi(P(K)) \to \pi(P_{\Delta}(D_1))\right) \neq
\ker \left( \pi(P(K)) \to \pi(P_{\Delta}(D_2))\right).\]
\end{proposition}
\begin{proof}
As observed in~\cite{GuthEtc},  it suffices to show that the inclusion induced maps
\[\pi_1(T^2) \to \pi_1(E_K),  \,  \pi_1(T^2) \to \pi_1(E_P), \, \pi_1(S^1 \times D^2) \to \pi_1(E_D), \,
\text{ and } \pi_1(S^1 \times D^2) \to \pi_1(E_\Delta)
\]
are all injective.  
A careful examination of the proof of \cite[Proposition 2.3]{GuthEtc} shows that,  while the authors assume that the winding number of $P$ is nonzero,  in fact the argument goes through under the weaker assumption that $\eta$ represents an infinite order class in $\pi(\Delta)$.
\end{proof}

Note the condition on the class of $\eta$ is automatically satisfied when the winding number of $P$ is nonzero for homological reasons,  and so we obtain the following from our examples of Section~\ref{sec:twist-spun}.

\begin{corollary}
\label{coro:non0winding}
	Let $P= P(U)$ be a pattern with nonzero winding number such that $P(U)$ is slice (respectively, ribbon) and let $J$ be  a nontrivial knot.
	Then $P(J \#\overline J)$ has infinitely many slice disks (respectively, ribbon disks) that are pairwise nonisotopic rel. boundary modulo local knotting. 
\end{corollary}

We would also like to be able to show that $P(K)$ has many LK-nonisotopic disks whenever $P(U)$ does.
Since $E_{P(U)}$ is not a subset of $E_{P(K)}$, the above argument does not quite work on the level of the fundamental group; instead, we use the Alexander module. 

\begin{proposition}
Let $P= P(U) \cup \eta$ be a pattern.  Suppose that $\Delta_1$ and $\Delta_2$ are slice disks for $P(U)$ that have
\[ \ker\left( \mathcal{A}(K) \to \mathcal{A}(\Delta_1)\right) \neq \ker\left( \mathcal{A}(K) \to \mathcal{A}(\Delta_2)\right).\]
Then for any slice disk $D$ for $K$,  the slice disks $P_{\Delta_1}(D)$ and $P_{\Delta_2}(D)$ for $P(K)$ have
\[ \ker\left( \mathcal{A}(P(K)) \to \mathcal{A}(P_{\Delta_1}(D))\right) 
\neq  \ker\left( \mathcal{A}(P(K)) \to \mathcal{A}(P_{\Delta_2}(D))\right).\]
\end{proposition}

We note that one can write down an Alexander module condition that guarantees that distinct disks for $K$ give distinct disks for $P(K)$, but such a result is subsumed by Proposition~\ref{prop:KtoP(K)}.

\begin{proof}
By \cite[Theorem 1]{Litherland},  there are isomorphisms 
\begin{align*}
\mathcal{A}(P(K)) &\cong \mathcal{A}(P(U)) \oplus \mathcal{A}(K)[t^w]\\
\mathcal{A}(P_{\Delta_i}(D)) &\cong \mathcal{A}(\Delta_i) \oplus \mathcal{A}(D)[t^w]
\end{align*}
such that the inclusion induced map $\mathcal{A}(P(K)) \to \mathcal{A}(P_{\Delta_i}(D))$ extends the inclusion induced map $\mathcal{A}(P(U)) \to \mathcal{A}(\Delta_i)$.
The desired result follows immediately. 
\end{proof}

The subsequent computations of Section~\ref{sec:Alex} will therefore give us the following corollary. 

\begin{corollary}
\label{coro:trefoil_pattern}
	Let $P$ be a pattern with $P(U)= T_{3,2} \#\overline T_{3,2}$.  
	Then for any slice (respectively, ribbon) knot $K$,  the satellite $P(K)$ has infinitely many slice disks (respectively, ribbon disks) that are pairwise nonisotopic rel. boundary modulo local knotting.  
\end{corollary}

\section{The Alexander module}\label{sec:Alex}

In this section,  we will show that the collection of ribbon disks $\left\{D_{2/(2k+1)}\right\}_{k\geq 0}$ for the square knot are pairwise LK-nonisotopic rel. boundary, as detected by the kernel of the inclusion-induced map $\mathcal{A}(Q) \to \mathcal{A}(D_{2/(2k+1)})$ on Alexander modules.  (Note that these disks are ribbon by~\cite[Theorem~1.6]{MeiZup_disks}, which relies on work of~\cite{zupanetal}.) 
This gives an alternate proof of the following special case of Theorem~\ref{thm:main_IRB}.

\begin{proposition}
\label{coro:alternate}
	The ribbon disks $\left\{D_{2/(2k+1)}\right\}_{k\geq0}$ bounded by the square knot are pairwise nonisotopic rel. boundary modulo local knotting, as detected by $\ker( \mathcal{A}(Q) \to \mathcal{A}(D_{2/(2k+1)}))$. 
\end{proposition}

\begin{proof}
	By Proposition~\ref{prop:module}, the kernels of the inclusion-induced maps on the Alexander modules are pairwise nonequal.
	It follows that the kernels of the inclusion-induced maps on the fundamental groups are pairwise nonequal.
	Hence, by \cite[Proposition~6.1]{MilPow_19}, the disks are pairwise LK-nonisotopic rel. boundary.
\end{proof}

We begin by giving a presentation for the Alexander module $\Aa(Q)$.  
Figure~\ref{fig:surface1} illustrates a genus-two Seifert surface $F$ for $Q$,  together with four simple closed curves $\alpha_1, \beta_1, \alpha_2, \beta_2$ representing a basis for $H_1(F)$.

Let $\widehat{\alpha_1}$, $\widehat{\beta_1}$,  $\widehat{\alpha_2}$,  and $\widehat{\beta_2}$ denote simple closed curves in $S^3 \setminus F$ representing a dual basis for $H_1(S^3 \setminus F)$.
(For example, these can be chosen to be positively oriented meridional curves for $\alpha_1$, $\beta_1$, $\alpha_2$, and $\beta_2$.)
Let $\kappa \colon F \to S^3 \setminus F$ be the map obtained by pushing off $F$ in the positive direction.
Then the induced map $\kappa_* \colon H_1(F) \to H_1(S^3 \setminus F)$ is given with respect to our fixed bases by $\kappa_*({v})= {v} \cdot A$,  where 
\[A=[\text{lk}(\kappa(x_i), x_j)]_{1 \leq i,j \leq 4}=
\begin{bmatrix} 1& 0 & \\ -1 & 1 \end{bmatrix} \oplus \begin{bmatrix}& -1 & 0 \\ & 1 & -1 \end{bmatrix}
.\]

The Alexander module $\mathcal{A}(Q)$ is generated by $a_1$, $b_1$, $a_2$, and $b_2$, which are the lifts of $\widehat{\alpha_1}$, $\widehat{\beta_1}$, $\widehat{\alpha_2}$, and $\widehat{\beta_2}$ to a preferred copy of a lift of $S^3 \setminus \nu(F)$ to the infinite cyclic cover $\widetilde{E_Q}$ of $E_Q= S^3\setminus\nu(Q)$.
These generators are subject to the relations encoded by the rows of the matrix  $tA-A^T$.
In particular, we see that 
\[ \mathcal{A}(Q)= \frac{ \Z[t^{\pm1}] }{ ( 1-t+t^2) }\langle a_1\rangle \oplus \frac{ \Z[t^{\pm1}] }{ ( 1-t+t^2) } \langle a_2\rangle \]
where $b_1= (1-t)a_1$,  $b_2=(1-t) a_2$.

Recall from Subsection~\ref{subsec:kernels} that the exterior $Z_{c/d} = B^4\setminus\nu(D_{c/d})$ is constructed from $Y_Q=S^3_0(Q)$ by attaching surface-framed 2--handles to $Y_Q$ along the link $\Ll_{c/d}$ on the closed surface-fiber.
This link can be isotoped to lie on $F\subset\widehat F$, so it represents a derivative link for $Q$.
(Technically, since $\Ll_{c/d}$ contains $pq$ curves, while $g(F) = (p-1)(q-1)$, it would be more precise to say that $\Ll_{c/d}$ \emph{contains} a derivative link.)
Any component $\gamma$ of $\Ll_{c/d}$ has the property that
 $[\widetilde{\kappa(\gamma)}]$ is an element of
$ \ker \left( \mathcal{A}(Q) \xrightarrow{\iota_*} \mathcal{A}(D_{c/d}) \right)$,
where $\widetilde{\kappa(\gamma)}$ denotes the lift of $\kappa(\gamma)$ to our fixed preferred lift of $S^3 \setminus \nu(F)$ in $\widetilde{E_Q}$.

\begin{figure}[h!]
	\begin{subfigure}{.5\textwidth}
		\centering
		\includegraphics[width=.9\linewidth]{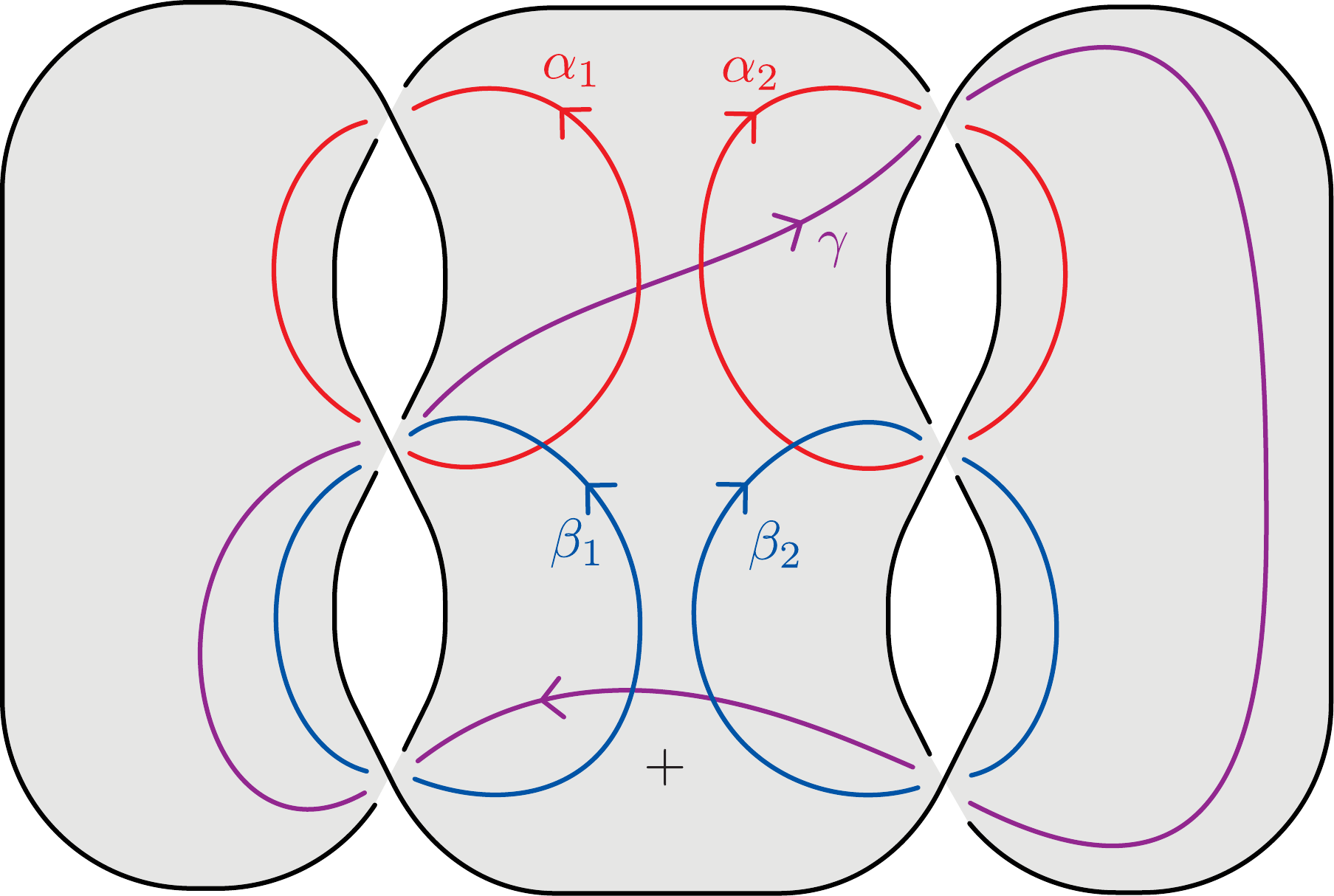}
		\caption{A basis for $H_1(F)$ and $\gamma$}
		\label{fig:surface1}
	\end{subfigure}%
	\begin{subfigure}{.5\textwidth}
		\centering
		\includegraphics[width=.9\linewidth]{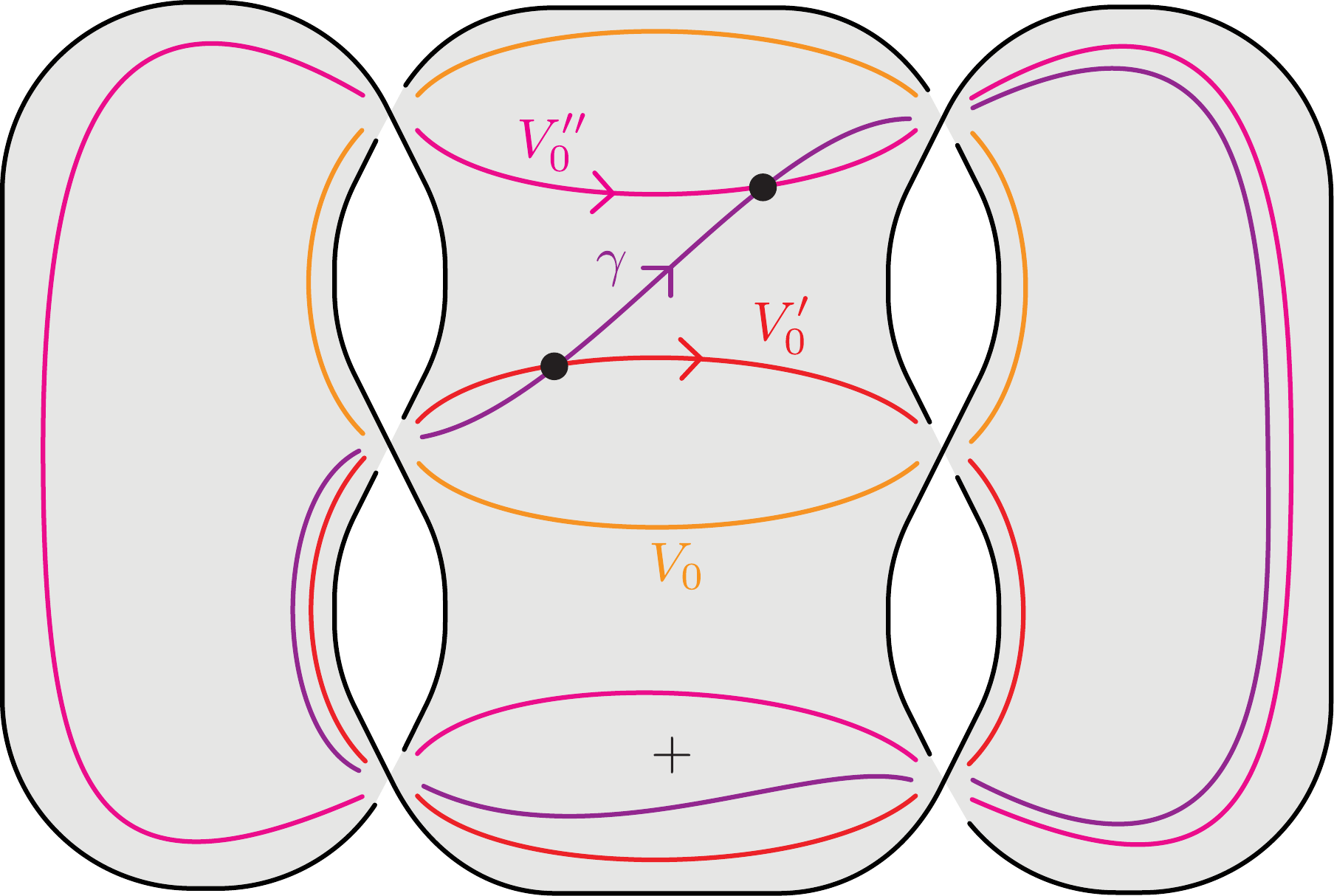}
		\caption{The link $\Ll_0$ and $\gamma$}
		\label{fig:surface2}
	\end{subfigure}%
	\caption{A Seifert surface $F$ for $Q$ with relevant curves.}
	\label{fig:surfaces}
\end{figure}

For an oriented curve $\gamma$ on $F$, we let $x_1=  \gamma \cdot \alpha_1$,  $y_1= \gamma \cdot \beta_1$,  $x_2= \gamma \cdot \alpha_2$,  $y_2= \gamma \cdot \beta_2$, and record these values as
$\begin{bmatrix} \gamma \end{bmatrix}_{\text{int}} = \begin{bmatrix} x_1  & y_1 & x_2 & y_2 \end{bmatrix}.$
The coordinate vector of $\gamma$ as an element of $H_1(F)$ in $(\alpha_1,\beta_1,\alpha_2,\beta_2)$--coordinates is given by 
\[ [\gamma] = \begin{bmatrix} x_1  & y_1 & x_2 & y_2 \end{bmatrix} \cdot \left( \begin{bmatrix} 0 & -1 \\ 1 & 0 \end{bmatrix} \oplus \begin{bmatrix} 0 & 1 \\ -1 & 0 \end{bmatrix} \right)
=\begin{bmatrix} y_1 & -x_1 & -y_2 & x_2 \end{bmatrix}.\]
The coordinate vector of $\widetilde{\kappa(\gamma)}$ as a linear combination of $a_1$, $b_1$, $a_2$, and $b_2$ is exactly the same as the coordinate vector of $\kappa(\gamma)$ as a linear combination of $\widehat{\alpha_1}$, $\widehat{\beta_1}$, $\widehat{\alpha_2}$, and $\widehat{\beta_2}$:
\[ [\widetilde{\kappa(\gamma)}]=
\begin{bmatrix} y_1 & -x_1 & -y_2 & x_2 \end{bmatrix} \cdot A= 
   \begin{bmatrix} x_1+y_1,  - x_1,  x_2+y_2,-x_2 \end{bmatrix}
\]
That is,  
\begin{align*}
[\widetilde{\kappa(\gamma)}]&=(x_1+y_1)a_1-x_1b_1+(x_2+y_2) a_2 - x_2 b_2\\
&= \left( x_1+y_1-x_1(1-t)\right)a_1 + \left( x_2+y_2 - x_2(1-t) \right) a_2 \\
&= (tx_1+y_1,  tx_2+y_2),
\end{align*}
where the last expression uses our previous identification of $ \mathcal{A}(Q)$ as $  \left(\Z[t^{\pm1}] / ( 1-t+t^2) \right)^2$.

\begin{example}
\label{ex:2/1}
	For a first example, let $\gamma$ denote the component $V_{2/1}$ of the derivative link $\mathcal{L}_{2/1}$ shown in purple in Figure~\ref{fig:surface1} (cf.~\cite[Figure~4]{MeiZup_disks}).
	Observe that
	$\gamma_{\text{int}}= \begin{bmatrix} 1& 0 & 1 & -1 \end{bmatrix}.$
	Using the above formulas, we obtain that
	$[\widetilde{\kappa(\gamma)}]= (t,t-1)$ is an element of 
	$P_0:= \ker \left( \mathcal{A}(Q) \xrightarrow{\iota_*} \mathcal{A}(D_{2/1}) \right).$
\end{example}

\begin{example}
	For a second example, we will explicitly determine an element of the kernel
	$$P_k:=\ker \left( \mathcal{A}(Q) \xrightarrow{\iota_*} \mathcal{A}(D_{2/(2k+1)}) \right)$$
	for each $k\geq 0$.
	Consider the link $\Ll_0 = V_0\cup V_0'\cup V_0''$ shown in Figure~\ref{fig:surface2}.
	Let $\tau_{\Ll_0}\colon F\to F$ denote a positive (right-handed) Dehn twist about the components of $\Ll_0$.
	Then, $\Ll_{2/(2k+1)} = \tau_{\Ll_0}(\Ll_{2/(2k-1)}) = \tau_{\Ll_0}^{k}(\Ll_{2/1})$,  as shown in~\cite[Subsection~5.2]{MeiZup_disks}. 

	Let $\gamma$ be the component of $\Ll_{2/1}$ analyzed in Example~\ref{ex:2/1} and replicated in Figure~\ref{fig:surface2}.
	Notice that $\gamma$ does not intersect $V_0$ and intersects each of $V_0'$ and $V_0''$ in a single point.
	It follows that 
	$$\begin{bmatrix} \tau_{\Ll_0}(\gamma) \end{bmatrix}_{\text{int}}= \begin{bmatrix}\gamma\end{bmatrix}_{\text{int}}+ \begin{bmatrix}V_0'\end{bmatrix}_{\text{int}}+ \begin{bmatrix}V_0''\end{bmatrix}_{\text{int}}.$$
	In fact,  since Dehn twisting $\gamma$ along $\mathcal{L}_{0}$ preserves the property of intersecting $\mathcal{L}_{0}$ in two points,  one in $V_0'$ and one in $V_0''$, if we let $\gamma_k = \tau_{\Ll_0}^k(\gamma)$,  we have that 
	$$\begin{bmatrix}\gamma_k\end{bmatrix}_{\text{int}} = \begin{bmatrix}\gamma\end{bmatrix}_{\text{int}}+ k\begin{bmatrix}V_0'\end{bmatrix}_{\text{int}}+k \begin{bmatrix}V_0''\end{bmatrix}_{\text{int}} = \begin{bmatrix} 2k+1 &- k & 2k+1 & -k-1 \end{bmatrix}.$$

	Using the formula preceding Example~\ref{ex:2/1}, we have that

\begin{eqnarray*}
		[\widetilde{\kappa(\gamma_k)}]&= \left((2k+1)t-k, (2k+1)t-(k+1) \right)
		\in P_k,
\end{eqnarray*} using our $(a_1, a_2)$--coordinates on $\Aa(Q)$.

We note for later that since $P_k$ is a submodule of $\mathcal{A}(Q)$,   we have
\[ w_k := \frac{-(2k+1)t+(k+1)}{3k^2+3k+1}[\widetilde{\kappa(\gamma_k)}]= \left( 1,  \frac{2k+1}{3k^2+3k+1} t+ \frac{3k^2+2k}{3k^2+3k+1}\right) \in P_{k} \otimes \Q.\]
\end{example}

We are now ready to prove the following.

\begin{proposition}
\label{prop:module}
	For $k_1,k_2 \geq 0$,  we have $P_{k_1} = P_{k_2}$ if and only if $k_1=k_2$.
\end{proposition}

\begin{proof}
Properties of the (rational) Blanchfield pairing imply that for any $c,d\in\Q$, we have:
\begin{align*}
\text{Bl}_Q((0,ct+d), (0,ct+d))&= \text{Bl}_{\overline T_{3,2}}(0,0) + \text{Bl}_{T_{3,2}}(ct+d,ct+d)\\
&= (ct+d)(ct^{-1}+d)\text{Bl}_{T_{3,2}}(1,1)\\
&= (cd(t-1+t^{-1})+cd+ c^2+d^2) \text{Bl}_{T_{3,2}}(1,1)\\
&=(c^2+ cd +d^2) \text{Bl}_{T_{3,2}}(1,1)\\
&= \left((c+d/2)^2+3d^2/4\right)\text{Bl}_{T_{3,2}}(1,1).
\end{align*}
Here the first equality comes from the fact that $\text{Bl}_Q= \text{Bl}_{\overline T_{3,2}} \oplus \text{Bl}_{T_{3,2}}$; 
the second  from the sesquilinearity of the Blanchfield pairing; 
the third and fifth  from algebraic manipulation; and 
the fourth from the fact that $\text{Bl}_{T_{3,2}}$ takes values in $\Q(t)/ \Q[t^{\pm1}]$ of the form $\frac{*}{t^2-t+1}$.
Since $\mathcal{A}({T_{3,2}})$ is cyclic,  $\text{Bl}_{T_{3,2}}(1,1) \neq 0 \in \Q(t)/\Q[t^{\pm1}]$.
(If it did vanish,  the entire Blanchfield pairing would be completely zero,  which would contradict its nonsingularity.)  
It follows that $\text{Bl}_Q((0,ct+d),(0,ct+d))=0$ if and only if $c=d=0$.   

Now suppose that $P_{k_1} = P_{k_2}$,  and so	
\[w_{k_1}- w_{k_2}= \left(0,\left( \frac{2k_1+1}{3k_1^2+3k_1+1}-\frac{2k_2+1}{3k_2^2+3k_2+1}\right) t + \left(\frac{3k_1^2+2k_1}{3k_1^2+3k_1+1} -\frac{3k_2^2+2k_2}{3k_2^2+3k_2+1}\right) \right)
\]
	is an element of $P_{k_1} \otimes \Q$.  Since all elements of $P_k \otimes \Q$ have vanishing Blanchfield pairing,  by the   previous paragraph we have $\frac{2k_1+1}{3k_1^2+3k_1+1}-\frac{2k_2+1}{3k_2^2+3k_2+1}=0$.  Since $\frac{2k+1}{3k^2+3k+1}$ is strictly decreasing as a function of $k$ when $k \geq 0$,  we have as desired that $k_1=k_2$. 
\end{proof} 

\begin{remark}
\label{rmk:upgrade_alex}
	It is not too difficult to understand the effect on $\Aa(Q)$ of the maps induced by the symmetries $\alpha$, $\beta$, and $\tau$ that generate $\textsc{Sym}(S^3,Q)$,   as  featured in the proof of Theorem~\ref{thm:main_I}.
	(For example, the map induced by $\tau$ is multiplication by $t^{\pm 1}$ in the second coordinate.)
	One can then show that for any element $ f \in \textsc{Sym}(S^3,Q)$,  we have that $f_*(P_{k_1}) \neq P_{k_2}$, giving an Alexander module-theoretic  proof that the disks $D_{2/(2k+1)}$ are LK-nonisotopic. 
\end{remark}

\bibliographystyle{amsalpha}
\bibliography{LK-isotopy.bib}

\end{document}